\documentclass[12pt]{article}

\usepackage[affil-it]{authblk}
      \usepackage{latexsym}
         \usepackage[reqno, namelimits, sumlimits]{amsmath}
         \usepackage{amssymb, amsfonts}
         \usepackage{amsthm}
\usepackage{bbold}

 \newtheorem{theorem}{Theorem}[section]
 
 \newtheorem{definition}[theorem]{Definition}
 
 \newtheorem{lemma}[theorem]{Lemma}
 
 \newtheorem{remark}[theorem]{Remark}
 
 \newtheorem{cor}[theorem]{Corollary}
 \newtheorem{pro}[theorem]{Proposition}

\title{An Approximation of Forward Self-Similar solutions to the 3D Navier-Stokes system}
\author{F.~Hounkpe\footnote{University of Oxford, Mathematical Institute, OxPDE, Oxford, UK,  email address: \texttt{hounkpe@maths.ox.ac.uk}}, G.~Seregin\footnote{University of Oxford, Mathematical Institute, OxPDE, Oxford, UK and St Petersburg Department of Steklov Mathematical Institute, RAS, Russia, email address: \texttt{seregin@maths.ox.ac.uk}}
}

\begin{document}
\maketitle

\begin{abstract}
    In this paper, we present two constructions of forward self-similar solutions to the $3$D incompressible Navier-Stokes system, as the singular limit of forward self-similar solutions to certain parabolic systems.
\end{abstract}

\setcounter{equation}{0}
\section{Introduction}

In \cite{Jia14}, H. Jia and V. \v{S}ver\'{a}k proved the existence of the so-called forward self-similar solution  $u\in C^{\infty}(\mathbb{R}^3\times ]0,\infty[)$ to the incompressible Navier-Stokes system
\begin{equation}\label{NS}
        \partial_t u + u \cdot \nabla u - \Delta u+ \nabla p = 0, \qquad
        {\rm div}\, u = 0
\end{equation} in $Q_+=\mathbb R^3\times ]0,\infty[$, satisfying the initial conditon
\begin{equation} \label{inidata}
    u|_{t=0} = u_0
\end{equation}
in $\mathbb R^3$, where $u_0$ is an arbitrary large $(-1)$-homogeneous divergence free vector valued field. By the definition, velocity $u$ is invariant with respect to the Navier-Stokes scaling, i.e., $\lambda u(\lambda x,\lambda^2 t) = u(x,t)$ for all $\lambda>0$. Such a problem can be reduced to the existence of a solution to the following stationary system
\begin{equation}\label{E1.8}
    -\Delta U + U\cdot \nabla U -\frac{x}{2}\cdot \nabla U - \frac{U}{2} + \nabla P = 0,\qquad
    {\rm div}\, U = 0
\end{equation} in $\mathbb R^3$
under the boundary condition at infinity
\begin{equation}\label{E1.7}
    |U(x)-u_0(x)| = o(|x|^{-1})\quad\mbox{as }|x|\to \infty.
\end{equation}
Then
$$
u(x,t) = \frac{1}{\sqrt{t}}U\left(\frac{x}{\sqrt{t}}\right)
$$
is a local energy weak solution to (\ref{NS}) and (\ref{inidata}) in the sense of Lemarie-Rieusset, see \cite{LemRie2002}, \cite{LemRie2016},   and also \cite{KiSer2007}.
An important idea of H. Jia and V. \v{S}ver\'{a}k is to use the possible instantaneous non-uniqueness of forward self-similar solutions in order to construct  different  weak Leray-Hopf solutions with the same $L_2$-initial data. In fact, they state a sufficient condition on the spectrum of the linearised problem ensuring non-uniqueness of forward self-similar solutions. So far, it is  an open problem whether the above condition holds for a certain $(-1)$-homogeneous initial data. However,
numerical experiments, see \cite{Gui17},  demonstrates that there are initial data for which the above spectrum condition is satisfied.

One of the main aim of the paper is to understand how known approximation schemes for solving the Cauchy problem for the Navier-Stokes system work in the case of forward self-similar solutions. There is a hope that it might help to get a better understanding of the non-uniqueness phenomenon. To this end, we replace the Navier-Stokes system (\ref{NS}) with the following ones:
\begin{equation}\label{Toy-Mod1}
    \partial_t u^\kappa - \Delta u^\kappa -\kappa \nabla {\rm div}\, u^\kappa+ u^\kappa\cdot \nabla +\frac{u^\kappa}{2} {\rm div}\,  u^\kappa= 0
\end{equation}
or
\begin{equation}\label{Toy-Mod2}
    \partial_t u^\kappa - \Delta u^\kappa-\kappa \nabla {\rm div}\,  u^\kappa +\left(u^\kappa \otimes u^\kappa +\frac{|u^\kappa|^2}{2} I_3\right) = 0
   \end{equation}
   and add up the Cauchy data
\begin{equation} \label{inidatakappa}
    u^\kappa|_{t=0} = u_0.\end{equation}
    Here, $\kappa\geq 0$ is a parameter and $I_3$ denotes the identity matrix.

For the profile $U^\kappa$, one has  then the following elliptic systems:
\begin{equation}\label{E1.5}
    -\Delta U^\kappa - \kappa\nabla {\rm div}\, U^\kappa + U^\kappa\cdot\nabla U^\kappa + \frac{U^\kappa}{2}{\rm div}\, U^\kappa - \frac{x}{2}\cdot\nabla U^\kappa - \frac{U^\kappa}{2} = 0\quad\mbox{in }\mathbb{R}^3,
\end{equation}
if we are solving \eqref{Toy-Mod1} or
\begin{equation}\label{E1.6}
    -\Delta U^\kappa - \kappa\nabla {\rm div}\, U^\kappa + {\rm div}\left( U^\kappa\otimes U^\kappa + \frac{|U^\kappa|^2}{2}I_3\right) - \frac{x}{2}\cdot\nabla U^\kappa - \frac{U^\kappa}{2} = 0\quad\mbox{in }\mathbb{R}^3,
\end{equation}
if we are solving \eqref{Toy-Mod2} instead. We require the following asymptotic on $U^\kappa$
\begin{equation}\label{E1.71}
    |U^\kappa(x)-u_0(x)| = o(|x|^{-1})\quad\mbox{as }|x|\to \infty.
\end{equation}
Then the corresponding solution to the Cauchy problems has the form
$$
u^\kappa(x,t) = \frac{1}{\sqrt{t}}U^\kappa\left(\frac{x}{\sqrt{t}}\right)
$$
for $x\in \mathbb R^3$ and $t>0$.

Now, we state the main results of the paper. Given $\kappa$, one can prove the existence of a forward self-similar solution to the Cauchy problem (\ref{Toy-Mod1}) and (\ref{inidatakappa}) or  (\ref{Toy-Mod2}) and (\ref{inidatakappa}), using the same method as in \cite{Jia14}. Our novelty here is that we use the notion of  global weak $L^{3,\infty}$-solution ($L^{3,\infty}$ denotes a weak Lebesgue space) which is slightly stronger than the notion of  weak Lemarie-Rieusset solutions. Global weak $L^{3,\infty}$-solutions have been introduced in \cite{Bark18} in the case of the Navier-Stokes equations. In order to present the corresponding definitions, we need  the semigroup $S_\kappa(t)$ associated to the Lam\'e system, i.e., $ v^\kappa(x,t)=S_\kappa(t)u_0(x)$, where
$ v^\kappa$ is a solution to the Cauchy problem:
\begin{equation}\label{lame}
	\partial_t v^\kappa-\Delta  v^\kappa-\kappa \nabla {\rm div}\, v^\kappa=0\end{equation}
in $Q_+$ and
\begin{equation}\label{semi-initial}
	v^\kappa(\cdot,0)=u_0(\cdot)\end{equation}
in $\mathbb R^3$.

For example, in  the case of
(\ref{Toy-Mod1}), the  definition of a global weak $L^{3,\infty}$-solution is as follows.
\begin{definition}
\label{globweak}
    We say that $u^\kappa$ is a global weak $L^{3,\infty}-$solution to
    the Cauchy problem (\ref{Toy-Mod1}) and (\ref{inidatakappa}) in $Q_+$ 
     if the function
    \begin{equation}
        w^\kappa = u^\kappa -  v^\kappa
    \end{equation}
    has the following properties:
    \begin{equation}\label{functionspaces}
    	\sup\limits_{0<t<T}\int\limits_{\mathbb R^3}|w^\kappa(x,t)|^2dx+\int\limits^T_0\int\limits_{\mathbb R^3}|\nabla w^\kappa|^2dxdt\leq C(T)<\infty
 \end{equation} for all $T>0$;

    \begin{equation}
        \partial_t w^\kappa - \Delta w^\kappa - \kappa \nabla {\rm div}\, w^\kappa + u^\kappa\cdot \nabla u^\kappa + \frac{u^\kappa}{2}{\rm div}\, u^\kappa = 0
    \end{equation}
    in the sense of distributions;

    the function
    \begin{equation}
        t \mapsto \int_{\mathbb{R}^3}w^\kappa(x,t)\cdot w(x) dx,
    \end{equation}
    is  continuous  at each $t\geq0$ for all $w\in L_2(\mathbb{R}^3)$;
    \begin{equation}
        \|w^\kappa(\cdot,t) \|_{L_2(\mathbb{R}^3)} \to 0\quad\mbox{as }t\to 0^+;
    \end{equation}

    for a.a. $t\in ]0,T[$, the  local energy inequality
    \begin{multline}
        \frac{1}{2}\int_0^t|u^\kappa(x,t)|^2 \phi(x,t)dx  + \int_0^t\int_{\mathbb{R}^3}\left(|\nabla u^\kappa|^2 + \kappa({\rm div}\, u^\kappa)^2\right)\phi(x,t)dx dt \\ \leq \int_0^t\int_{\mathbb{R}^3}\frac{|u^\kappa|^2}{2}(\partial_t\phi + \Delta \phi)dx dt + \int_0^t\int_{\mathbb{R}^3}(\frac{|u^\kappa|^2}{2} - \kappa{\rm div}\, u^\kappa)u^\kappa\cdot \nabla \phi dx dt.
    \end{multline}
is valid for each non-negative test function $\phi\in C^\infty_0(Q_+)$.
\end{definition}

 \begin{remark}\label{Toy2}
We have also an analogous definition for the Cauchy problem (\ref{Toy-Mod2}) and (\ref{inidatakappa}).
\end{remark}
\begin{remark} \label{turbulent} It is easy to show that $w^\kappa$ is a turbulent solution in the Leray sense, see \cite{Leray1934}. In other words, for all $t\in [0,T]$,
\begin{equation}\label{turb1}
\frac 12\int\limits_{\mathbb R^3}|w^\kappa(x,t)|^2dx+\int\limits_s^t	\int\limits_{\mathbb R^3}(|\nabla w^\kappa|^2+\kappa|{\rm div}\,w^k|^2)dxdt'\leq
\end{equation}
$$\leq\frac 12\int\limits_{\mathbb R^3}|w^\kappa(x,s)|^2dx	+
\int\limits_s^t	\int\limits_{\mathbb R^3}(v^\kappa\otimes w^\kappa+v^\kappa\otimes v^\kappa):\nabla w^kdxdt'+
$$$$+\int\limits_s^t	\int\limits_{\mathbb R^3}\frac 12v^\kappa\cdot w^\kappa{\rm div}\,w^\kappa dxdt'
$$
for a.a. $s\in [0,T]$, including $s=0$.
\end{remark}
Now, we start to formulate  our results with the following statement.
\begin{theorem}\label{existenceglobalweaksol}
Assume that $u_0\in L^{3,\infty}(\mathbb{R}^3)$. There exists at least one global weak $L^{3,\infty}-$ solution $u^\kappa$ to the Cauchy problem \eqref{Toy-Mod1} and \eqref{inidatakappa}.
Moreover,  the  global energy estimate
\begin{multline}\label{E1.21}
    \|w^{\kappa}(\cdot,t)\|^2_{L_2(\mathbb{R}^3)} + \int_0^t\int_{\mathbb{R}^3}|\nabla w^{\kappa}(x,s)|^2dxds\\ + \kappa\int_0^t \int_{\mathbb{R}^3}|{\rm div}\, w^{\kappa}|^2 dx ds \leq c_0t^{\frac{1}{2}}\left(\|u_0\|^2_{L^{3,\infty}(\mathbb{R}^3)} + \|u_0\|^4_{L^{3,\infty}(\mathbb{R}^3)} \right),
\end{multline}
holds for all $\kappa, t>0$ and for  an absolute positive constant $c_0$.
\end{theorem}
The proof of the above theorem is based on ideas developed in \cite{Bark18} and is given in Appendix I.

As to a forward self-similar solutions, since
 it is unknown whether the solution  constructed in the above theorem is unique, there is no guarantee that a $(-1)$-homogeneous initial data produces a scale invariant solution. However, we are able to prove the following result.

\begin{theorem}\label{existenceself-similar}
Let $u_0\in C^{\infty}(\mathbb{R}^3\setminus \{0\})$ such that $\lambda u_0(\lambda x) = u_0(x)$ for all $\lambda >0$. Then, given $\kappa\geq 0$, there exists a smooth solution $U^\kappa$  to the boundary value problem (\ref{E1.5}) and (\ref{E1.71}) satisfying the decay estimates:
$$
|\partial^{\alpha}\left( U^\kappa(x) - V^\kappa(x) \right)| \leq \frac{C(\alpha,\kappa,u_0)}{(1 + |x|)^{3+|\alpha|}},\qquad V^\kappa(x):=v^\kappa(x,1),
$$
for all $\alpha\in \mathbb{N}^3$ (with $|\alpha| = \alpha_1 + \alpha_2 + \alpha_3$) and for all $x\in \mathbb R^3$.

Moreover, $u^\kappa(x,t)=\frac 1{\sqrt t}U^\kappa(\frac x{\sqrt t})$ is a global weak $L^{3,\infty}-$ solution to system \eqref{Toy-Mod1} with initial data $u_0$, see (\ref{inidatakappa}). 
\end{theorem}

\begin{remark}
    The result in the above theorem holds also true for the boundary value problem  \eqref{E1.6} and (\ref{E1.71}); with some simplifications in the computations due to the divergence structure of the nonlinearity. We do not give  the details of the computations for the sake of brevity.
\end{remark}
The proof of Theorem \ref{existenceself-similar} can be done along the lines of the paper \cite{Jia14} and it is given in Appendix II.

Now, we are able to state the main results of the paper. It is about behaviour of solutions to the boundary value problem (\ref{E1.5}) and (\ref{E1.71}) as $\kappa\to \infty$. To this end, let us make a simple remark: if ${\rm div}\,u_0=0$, then $v^\kappa(x,t)=v(x,t)=S(t)u_0(x)$, where $S(t)$ is a semigroup associated with the usual heat equation.
\begin{theorem}\label{UniformEstimates}
Let $u_0\in C^{\infty}(\mathbb{R}^3\setminus \{0\})$ such that $\lambda u_0(\lambda x) = u_0(x)$ for all $\lambda >0$ and ${\rm div}\, u_0 = 0$. Let $W^\kappa(x)=U^\kappa(x)-V(x)$, where $U^\kappa$ is a smooth solution to the boundary value problem (\ref{E1.5}) and (\ref{E1.71}) constructed in Theorem \ref{existenceself-similar} and $V(x)=v(x,1)$. Then the following estimates are valid:
\begin{equation}\label{E1.22}
    \|W^{\kappa}\|^2_{L_2(\mathbb{R}^3)} + \|\nabla W^{\kappa}\|^2_{L_2(\mathbb{R}^3)} + \kappa\|{\rm div}\, U^{\kappa}\|^2_{L_2(\mathbb{R}^3)} \leq c\left( \|u_0\|^2_{L^{3,\infty}(\mathbb{R}^3)} + \|u_0\|^4_{L^{3,\infty}(\mathbb{R}^3)} \right),
\end{equation}
and
\begin{equation}\label{E1.23}
    \|\kappa {\rm div}\, U^{\kappa}\|^2_{L_2(\mathbb{R}^3)} \leq c\left( \|u_0\|^2_{L^{3,\infty}(\mathbb{R}^3)} + \|u_0\|^4_{L^{3,\infty}(\mathbb{R}^3)} \right)^2,
\end{equation}
where $c>0$ is a universal constant.

Moreover,
\begin{equation}\label{H2est}
\int_{\mathbb{R}^3}|\nabla^2 W^{\kappa}|^2 dx + \kappa^2 \int_{\mathbb{R}^3}|\nabla {\rm div}\, W^{\kappa}|^2 dx \leq C(\|u_0\|_{L^{3,\infty}(\mathbb{R}^3)}).
\end{equation}

\end{theorem}
\begin{theorem}\label{Convergence} Under the assumptions of Theorem \ref{UniformEstimates},
 there exists a subsequence, still indexed by $\kappa$, such that
\begin{equation}\label{E1.24}
    W^\kappa\to  W,\qquad \nabla W^\kappa\to\nabla W, \qquad\nabla^2 W^\kappa\to\nabla^2 W  \end{equation}
and
\begin{equation}
	\label{pressureconvergence}
	\kappa\,{\rm div}\,w^k\rightharpoonup P\qquad \kappa\nabla {\rm div}\,w^k\rightharpoonup\nabla P\end{equation}
in $L_2(\mathbb R^3)$, where limiting functions $U=V+W$  and $P$ have the following properties:

(i)
\begin{equation}\label{DecayNavierStokes}
|\partial^{\alpha}W(x)| \leq \frac{}{}\frac{C(\alpha,u_0)}{(1 + |x|)^{3+|\alpha|}};
	\end{equation} for all $\alpha\in \mathbb{N}^3$ (with $|\alpha| = \alpha_1 + \alpha_2 + \alpha_3$) and $x\in\mathbb R^3$;

(ii) the function $u(x,t) = \frac{1}{\sqrt{t}}U(\frac{x}{\sqrt{t}})$ (together with $p(x,t)=\frac 1tP(\frac{x}{\sqrt{t}})$) is a global weak $L^{3,\infty}-$solution (in the sense of Barker-Seregin-\v{S}ver\'ak \cite{Bark18}) to the incompressible Navier-Stokes system \eqref{NS} with initial data $u_0$.
\end{theorem}
\begin{remark}
    Once more, the above result holds also for system \eqref{E1.6}.
\end{remark}

\setcounter{equation}{0}
\section{Preliminaries}

Set $Q_{T_1,T_2}=\Omega\times ]T_1,T_2[$, where $\Omega$ is a domain in $\mathbb R^3$.  The  notation for mixed Lebesgue and Sobolev spaces is as follows:
$L_{m,n}(Q_{T_1,T_2}):= L_n(T_1,T_2;L_m(\Omega))$, the Lebesgue space with the norm
\[
\|v\|_{m,n,Q_{T_1,T_2}} = \begin{cases}
\left( \int_{T_1}^{T_2}\|v(\cdot,t)\|^n_{L_m(\Omega)}dt\right)^{1/n},\quad & 1\leq n< \infty\\
\displaystyle{\rm ess sup}_{t\in (T_1,T_2)}\|v(\cdot,t)\|_{L_m(\Omega)},\quad & n= \infty,
\end{cases}
\]
\[ L_m(Q_{T_1,T_2})=L_{m,m}(Q_{T_1,T_2}), \quad \|v\|_{m,m,Q_{T_1,T_2}} = \|v\|_{m,Q_{T_1,T_2}}; \]
$W^{1,0}_{m,n}(Q_{T_1,T_2})$, $W^{2,1}_{m,n}(Q_{T_1,T_2})$ are the Sobolev spaces with mixed norm,
\[ W^{1,0}_{m,n}(Q_{T_1,T_2}) = \left\{ v,\nabla v \in L_{m,n}(Q_{T_1,T_2}) \right\}, \]
\[ W^{2,1}_{m,n}(Q_{T_1,T_2}) = \left\{ v,\nabla v, \nabla^2 v, \partial_t v \in L_{m,n}(Q_{T_1,T_2}) \right\}, \]
\[ W^{1,0}_m(Q_{T_1,T_2}) = W^{1,0}_{m,m}(Q_{T_1,T_2}), \quad W^{2,1}_m(Q_{T_1,T_2}) = W^{2,1}_{m,m}(Q_{T_1,T_2}).\]

In this work $L^{p,\infty}(\Omega)$ ($0<p<\infty$) stands for the \textit{weak} $L_p(\Omega)$ space of functions $f$ such that
$$
  \|f\|_{L^{3,\infty}(\Omega)}:= \sup_{\gamma > 0} \left\{\gamma |\{ x \in \Omega: |f(x)|> \gamma \}|^{\frac{1}{p}}\right\} < \infty.$$
It is not difficult to show that $L_p(\Omega) \subset L^{p,\infty}(\Omega)$, and this holds for $\Omega$ with finite measure or not (see \cite{Grafa08} for more properties of this function space). Those are a special case of the Lorentz spaces $L^{p,q}(\Omega)$ (with $0<p,q \leq \infty$) which consist, when $p,q \neq \infty$, of functions $f$ such that
\[
\|f\|_{L^{p,q}(\Omega)} := p^{\frac{1}{q}}\left(\int_0^{\infty} s^{q-1}|\{ x \in \Omega: |f(x)|> s \}|^{\frac{q}{p}}\right)^{\frac{1}{q}} < \infty,
\]
with $L^{\infty,q} = \{0\}$ whenever $0<q<\infty$ and $L^{p,p} = L_p$ for every $0<p\leq \infty$.

We use $c$ or $C$ to denote an absolute constant and we write $C(A,B,\ldots)$ when the constant depends on the parameters $A,B,\ldots$.


We record some estimates for the solutions to the Cauchy problem for the Lam\'e system.
\begin{pro}\label{Prop1.1}
Let $S_{\kappa}(t)$ be the semigroup associated to the Lam\'e system, see the Cauchy problem (\ref{lame}) and  (\ref{semi-initial}). Then
\begin{equation}\label{E1.11}
    \|S_{\kappa}(t)u_0\|_{L^{3,\infty}(\mathbb{R}^3)} \leq c\|u_0\|_{L^{3,\infty}(\mathbb{R}^3)},\quad\forall t\geq 0.
\end{equation}
For $1\leq s_1 \leq s$, we have
\begin{equation}\label{E1.12}
    \|S_{\kappa}(t)u_0\|_{L_s(\mathbb{R}^3)} \leq c(s,s_1)\left[1 + (1+\kappa)^{-\frac{1}{l}} \right] t^{-\frac{1}{l}}\|u_0\|_{L_{s_1}(\mathbb{R}^3)},
\end{equation}
where
\[
\frac{1}{l} = \frac{3}{2}\left( \frac{1}{s_1} - \frac{1}{s} \right).
\]
\end{pro}

\begin{proof}
The classical Calderon-Zygmund combined with real interpolation methods allow us to get the existence of a unique function $q_0$ up to a constant (it doesn't matter here since we are interested in the gradient of $q_0$) such that $\Delta q_0 = {\rm div}\, u_0$ and
\begin{equation}\label{E1.13}
    \|\nabla q_0\|_{L^{3,\infty}(\mathbb{R}^3)} \leq c\|u_0\|_{L^{3,\infty}(\mathbb{R}^3)}.
\end{equation}
Set $u_0^{(1)} := \nabla q_0$ and $u_0^{(0)} := u_0 - u^{(1)}_0$; notice that ${\rm div}\, u^{(0)}_0 = 0$ and ${\rm curl}\, u_0^{(1)} = 0$ by definition.
It is easy to check that $\Delta v^1=\nabla {\rm div}\,v^1$ and ${\rm div}\,v^0=0$, where
$$\partial_tv^1-(1+\kappa)\Delta v^1=0, \qquad v^1|_{t=0}=u^{(1)}_0$$
and
$$\partial_tv^0-\Delta v^0=0, \qquad v^0|_{t=0}=u^{(0)}_0.$$
Moreover,
\begin{equation}\label{solformula}
	S_\kappa(t)u_0(\cdot)=v^0(\cdot,t)+v^1(\cdot,t)=S(1+\kappa)u^{(1)}_0+S(1)u^{(0)}_0.
\end{equation}
Finally, from \eqref{E1.13} and (\ref{solformula}), from the convolution structure of the heat potential and from Young's inequality for weak type spaces (see \cite{Grafa08} Theorem 1.2.13), we get \eqref{E1.11}.

For \eqref{E1.12}, we see have instead
\[
\|u_0^{(0)}\|_{L_{s_1}(\mathbb{R}^3)} + \|u_0^{(1)}\|_{L_{s_1}(\mathbb{R}^3)} \leq c(s_1)\|u_0\|_{L_{s_1}(\mathbb{R}^3)}.
\]
And once again, using the representation formula for solutions of the heat equation, Young's inequality and scaling arguments, we have \eqref{E1.12}. This concludes the proof.
\end{proof}


\setcounter{equation}{0}
\section{Uniform Estimates}

In this section we are going to prove Theorem
\ref{UniformEstimates}.

Since estimate (\ref{E1.22}) follows from estimate (\ref{E1.21}),
our main goal now is to prove the following bound
\begin{equation} \label{important}
    \int_0^T\int_{\mathbb{R}^3}|\kappa {\rm div}\, w^{\kappa}|^2 dx ds \leq cT^{\frac{1}{2}}\left(\|u_0\|^2_{L^{3,\infty}(\mathbb{R}^3)} + \|u_0\|^4_{L^{3,\infty}(\mathbb{R}^3)} \right)^2
\end{equation}
for all $T>0$ and for an absolute constant $c>0$, see notation in Section 1.
Its proof is divided into three parts.

\paragraph{Part I [A priori estimates].} We focus here, only, on system \eqref{Toy-Mod1} since things are the same for system \eqref{Toy-Mod2} with some simplifications due to the divergence structure of the non-linearity. We recall that $u^{\kappa} = w^{\kappa} + v$ and
\begin{equation*}
    \begin{gathered}
    U^{\kappa}(x) := u^{\kappa}(x,1)\quad\mbox{and}\quad u^{\kappa}(x,t) = \frac{1}{\sqrt{t}}U^{\kappa}\left( \frac{x}{\sqrt{t}} \right)\\
    v(x,t) = S(t)u_0(x) = \frac{1}{\sqrt{t}}V\Big(\frac x{\sqrt{t}} \Big)\quad\mbox{and}\quad W^{\kappa} := U^{\kappa} - V
    \end{gathered}
\end{equation*}
(notice that ${\rm div}\, v = 0$ in $Q_+:=\mathbb{R}^3\times ]0,\infty[$ since ${\rm div}\, u_0 = 0$) and
\begin{multline}\label{E2.10}
    \partial_t w^{\kappa} - \Delta w^{\kappa} - \kappa\nabla {\rm div}\, w^{\kappa} = -( w^{\kappa}\cdot \nabla w^{\kappa} + \frac{w^{\kappa}}{2}{\rm div}\, w^{\kappa}) - (v\cdot \nabla w^{\kappa} + \frac{v}{2}{\rm div}\, w^{\kappa} \\ + w^{\kappa}\cdot \nabla v+ v\cdot \nabla v) \end{multline}
in $Q_+$.

Now, we introduce the functions $w^{\kappa,1}$, $\hat{w}^{\kappa,i}$ and $p^{\kappa}_i$ ($i=1,2,3$) as solutions to the following Cauchy problems

\begin{equation}\label{E2.13}
    \left\{
    \begin{gathered}
    \partial_t w^{\kappa,1} - \Delta w^{\kappa,1} - \kappa \nabla {\rm div}\, w^{\kappa,1} =  -( w^{\kappa}\cdot \nabla w^{\kappa} + \frac{w^{\kappa}}{2}{\rm div}\, w^{\kappa})\quad \mbox{in } Q_+
    \\
    w^{\kappa,1}|_{t=0} = 0 \quad \mbox{in } \mathbb{R}^3,
    \end{gathered}
    \right.
\end{equation}

\begin{equation}\label{E2.14}
    \left\{
    \begin{gathered}
    \partial_t \hat{w}^{\kappa,1} - \Delta \hat{w}^{\kappa,1} + \nabla p^{\kappa}_1 =  -( v\cdot \nabla w^{\kappa} + \frac{1}{2}{\rm div}\, (v\otimes w^{\kappa}))\quad \mbox{in } Q_+
    \\
    {\rm div}\, \hat{w}^{\kappa,1} = 0 \quad \mbox{in } Q_+
     \\
    \hat{w}^{\kappa,1}|_{t=0} = 0 \quad \mbox{in }\mathbb{R}^3,
    \end{gathered}
    \right.
\end{equation}

\begin{equation}\label{E2.15}
    \left\{
    \begin{gathered}
    \partial_t \hat{w}^{\kappa,2} - \Delta \hat{w}^{\kappa,2} + \nabla p^{\kappa}_2 =  -\frac{1}{2} w^{\kappa}\cdot \nabla v\quad \mbox{in } Q_+
    \\
    {\rm div}\, \hat{w}^{\kappa,2} = 0 \quad \mbox{in } Q_+
    \\
    \hat{w}^{\kappa,2}|_{t=0} = 0 \quad \mbox{in }\mathbb{R}^3,
    \end{gathered}
    \right.
\end{equation}
and
\begin{equation}\label{E2.16}
    \left\{
    \begin{gathered}
    \partial_t \hat{w}^{\kappa,3} - \Delta \hat{w}^{\kappa,3} + \nabla p^{\kappa}_3 =  -v\cdot \nabla v\quad \mbox{in } Q_+
    \\
    {\rm div}\, \hat{w}^{\kappa,3} = 0 \quad \mbox{in } Q_+
    \\
    \hat{w}^{\kappa,3}|_{t=0} = 0 \quad \mbox{in }\mathbb{R}^3.
    \end{gathered}
    \right.
\end{equation}
The proof for the unique solvability of the above Cauchy problems  (\ref{E2.13})-(\ref{E2.16}) in the energy class
\begin{multline}\label{E2.11}
    \sup_{0<t<T}\left(\|w^{\kappa,1}(\cdot,t)\|^2_{L_2(\mathbb{R}^3)} + \|\hat{w}^{\kappa,i}(\cdot,t)\|^2_{L_2(\mathbb{R}^3)} \right)\\ + \int_0^T\int_{\mathbb{R}^3}\left( |\nabla w^{\kappa,1}(x,t)|^2 + |\nabla \hat{w}^{\kappa,i}(x,t)|^2 \right)dx dt < +\infty,\forall T>0
\end{multline}
and
\begin{equation}\label{E2.12}
   \int_0^T\int_{\mathbb{R}^3} |p^{\kappa}_i(x,t)|^2 dx dt < \infty, \forall T>0~(i=1,2,3)
\end{equation}
is more or less standard, see estimates below and Part III.


We now derive some uniform (in $\kappa$) $L_2$-estimates for the pressure functions $p^{\kappa}_i$ in $Q_T:=\mathbb{R}^3\times ]0,T[$. From \eqref{E2.14}, we have
$$
-\Delta p^{\kappa}_1 = {\rm div}\,{\rm div}\,(w^{\kappa}\otimes v + \frac{1}{2}v\otimes w^{\kappa}),
$$
thus
\begin{align*}
    \|p^{\kappa}_1(\cdot,t)\|_{L_2(\mathbb{R}^3)} &\leq c \||w^{\kappa}(\cdot,t)|\cdot|v(\cdot,t)|\|_{L_2(\mathbb{R}^3)}\\
    &\leq c \|w^{\kappa}(\cdot,t)\|_{L^{6,2}(\mathbb{R}^3)}\|v(\cdot,t)\|_{L^{3,\infty}(\mathbb{R}^3)}\\
    &\leq c\|u_0\|_{L^{3,\infty}}\|\nabla w^{\kappa}(\cdot,t)\|_{L_2(\mathbb{R}^3)},
\end{align*}
for a.e. $t\in ]0,T[$. Consequently, taking into account Theorem \ref{existenceglobalweaksol}, we find
\begin{equation}\label{E2.17}
    \int_0^T\int_{\mathbb{R}^3}|p^{\kappa}_1(x,t)|^2 dx dt \leq cT^{\frac{1}{2}}\|u_0\|^2_{L^{3,\infty}(\mathbb{R}^3)}\left(\|u_0\|^2_{L^{3,\infty}(\mathbb{R}^3)} + \|u_0\|^4_{L^{3,\infty}(\mathbb{R}^3)} \right),
\end{equation}
for all $T>0$.

Next, we have
\[
-\Delta p^{\kappa}_2 = {\rm div}\,(\frac{1}{2} w^{\kappa}\cdot \nabla v),
\]
and thus
\begin{align*}
    \|\nabla p^{\kappa}_2(\cdot,t)\|_{L^{\frac{6}{5},2}(\mathbb{R}^3)} &\leq c \|w^{\kappa}(\cdot,t)\|_{L_2(\mathbb{R}^3)}\|\nabla v(\cdot,t)\|_{L^{3,\infty}(\mathbb{R}^3)}\\
    &\leq \frac{c}{t^{\frac{1}{2}}}\|w^{\kappa}(\cdot,t)\|_{L_2(\mathbb{R}^3)} \|u_0\|_{L^{3,\infty}(\mathbb{R}^3)}\quad \mbox{for a.e. }t\in ]0,T[,
\end{align*}
where the last inequality is a consequence of Young's inequality for weak type spaces (see \cite{Grafa08}, Theorem 1.2.13) and $L_1(\mathbb{R}^3)$-estimate for the gradient of the heat kernel. Consequently, changing the $p^\kappa_2$ in a suitable way, applying  Sobolev embedding and Theorem \ref{existenceglobalweaksol}, we find
\begin{equation}\label{E2.18}
    \int_0^T\int_{\mathbb{R}^3}|p^{\kappa}_2(x,t)|^2 dx dt \leq cT^{\frac{1}{2}}\|u_0\|^2_{L^{3,\infty}(\mathbb{R}^3)}\left(\|u_0\|^2_{L^{3,\infty}(\mathbb{R}^3)} + \|u_0\|^4_{L^{3,\infty}(\mathbb{R}^3)} \right),
\end{equation}
for all $T>0$. Finally, we have
\[
-\Delta p^{\kappa}_3 = {\rm div}\,{\rm div}\,(v\otimes v),
\]
thus
\begin{align*}
    \|p^{\kappa}_3(\cdot,t)\|_{L_2(\mathbb{R}^3)} &\leq c \|v(\cdot,t)\|^2_{L_4(\mathbb{R}^3)}\\
    &\leq \frac{c}{t^{\frac{1}{4}}}\|u_0\|^2_{L^{3,\infty}(\mathbb{R}^3)}\quad \mbox{for a.e. }t\in ]0,T[,
\end{align*}
where again the last estimate follows from the convolution structure of the heat potential and the corresponding inequalities. Consequently,
\begin{equation}\label{E2.19}
    \int_0^T\int_{\mathbb{R}^3}|p^{\kappa}_3(x,t)|^2 dx dt \leq c T^{\frac{1}{2}}\|u_0\|^4_{L^{3,\infty}},
\end{equation}
for all $T>0$. In conclusion, we have obtained that
\begin{equation}\label{E2.20}
    \sum_{i=1}^3\int_0^T\int_{\mathbb{R}^3}|p^{\kappa}_i(x,t)|^2 dx dt \leq cT^{\frac{1}{2}}\|u_0\|^2_{L^{3,\infty}(\mathbb{R}^3)}\left(\|u_0\|^2_{L^{3,\infty}(\mathbb{R}^3)} + \|u_0\|^4_{L^{3,\infty}(\mathbb{R}^3)} \right),
\end{equation}
for all $T>0$.

Going back once more to Theorem \ref{existenceglobalweaksol}, we see that
\[
\|W^{\kappa}\|^2_{L_2(\mathbb{R}^3)} + \|\nabla W^{\kappa}\|^2_{L_2(\mathbb{R}^3)} \leq c\left(\|u_0\|^2_{L^{3,\infty}(\mathbb{R}^3)} + \|u_0\|^4_{L^{3,\infty}(\mathbb{R}^3)} \right);
\]
therefore, by Sobolev embedding, we have
\[
\|W^{\kappa}\|_{L_3(\mathbb{R}^3)} \leq c \left(\|u_0\|^2_{L^{3,\infty}(\mathbb{R}^3)} + \|u_0\|^4_{L^{3,\infty}(\mathbb{R}^3)} \right)^{\frac{1}{2}},
\]
where $c>0$ is an absolute constant independent of $\kappa$. Consequently, we have
\begin{equation}\label{E2.21}
    \sup_{0<t<\infty}\|w^{\kappa}(\cdot,t)\|_{L_3(\mathbb{R}^3)} \leq c \left(\|u_0\|^2_{L^{3,\infty}(\mathbb{R}^3)} + \|u_0\|^4_{L^{3,\infty}(\mathbb{R}^3)} \right)^{\frac{1}{2}}.
\end{equation}
From the latter estimate, we have (thanks again to Theorem \ref{existenceglobalweaksol}) that
\begin{equation}\label{E2.22}
    \|w^{\kappa}\cdot \nabla w^{\kappa} + \frac{w^{\kappa}}{2}{\rm div}\, w^{\kappa}\|^2_{L_{\frac{6}{5},2}(\mathbb{R}^3\times ]0,T[)} \leq c T^{\frac{1}{2}} \left(\|u_0\|^2_{L^{3,\infty}(\mathbb{R}^3)} + \|u_0\|^4_{L^{3,\infty}(\mathbb{R}^3)} \right)^2.
\end{equation}

Now, let us introduce the functions (whose existence's justification is similar to the one of $\hat{w}^{\kappa,2}$ and $p^{\kappa}_2$)
\[
z^{\kappa} \in L_{2,\infty}(Q_T)\cap W^{1,0}_2(Q_T)\quad \forall T>0,
\]
and
\[
q^{\kappa} \in L_2(Q_T)\quad \forall T>0,
\]
such that
\begin{equation}\label{E2.23}
    \left\{
    \begin{gathered}
    \partial_t z^{\kappa} - \Delta z^{\kappa} + \nabla q^{\kappa} =  -(w^{\kappa}\cdot \nabla w^{\kappa} + \frac{w^{\kappa}}{2}{\rm div}\, w^{\kappa})\quad \mbox{in }Q _+\\
    {\rm div}\, z^{\kappa} = 0 \quad \mbox{in } Q_+ \\
    z^{\kappa}|_{t=0} = 0 \quad \mbox{in }\mathbb{R}^3
    \end{gathered}
    \right.
\end{equation}
From the equation
\[
-\Delta q^{\kappa} = {\rm div}\,(w^{\kappa}\cdot \nabla w^{\kappa} + \frac{w^{\kappa}}{2}{\rm div}\, w^{\kappa}),
\]
estimate \eqref{E2.22} and thanks to Sobolev embedding, we see that the pressure $q^\kappa$ can be chosen so that,  for  all $T>0$, the following inequality is valid:
\begin{equation}\label{E2.24}
    \int_0^T\int_{\mathbb{R}^3}|q^{\kappa}(x,t)|^2 dx dt \leq c T^{\frac{1}{2}} \left(\|u_0\|^2_{L^{3,\infty}(\mathbb{R}^3)} + \|u_0\|^4_{L^{3,\infty}(\mathbb{R}^3)} \right)^2.
\end{equation}
Now, we set $\bar{w}^{\kappa,1} := w^{\kappa,1} - z^{\kappa}$ and we see that
\begin{equation}
    \left\{
    \begin{gathered}
    \partial_t \bar{w}^{\kappa,1} - \Delta \bar{w}^{\kappa,1} - \kappa \nabla {\rm div}\, \bar{w}^{\kappa,1} =  \nabla q^{\kappa} \quad \mbox{in }Q_+ \\
    \bar{w}^{\kappa,1}|_{t=0} = 0 \quad \mbox{in }\mathbb{R}^3
    \end{gathered}
    \right.
\end{equation}
Similarly to what has been done in the proof of Theorem \ref{existenceglobalweaksol}, we get for all $t>0$ the energy estimate:
\begin{multline*}
    \frac{1}{2}\int_{\mathbb{R}^3}|\bar{w}^{\kappa,1}(x,t)|^2 dx + \int_0^t\int_{\mathbb{R}^3}|\nabla \bar{w}^{\kappa,1}(x,s)|^2dx ds\\ + \kappa\int_0^t\int_{\mathbb{R}^3}|{\rm div}\, \bar{w}^{\kappa,1}(x,s)|^2 dx ds \leq \left(\int_0^t\int_{\mathbb{R}^3} |q^{\kappa}|^2 dx ds \right)^{\frac{1}{2}}\left(\int_0^t\int_{\mathbb{R}^3}|{\rm div}\, \bar{w}^{\kappa,1}|^2 dx ds\right)^{\frac{1}{2}}.
\end{multline*}
We multiply both sides of the above inequality by $\kappa$ and use Young's inequality to obtain
\begin{multline*}
    \kappa\int_{\mathbb{R}^3}|\bar{w}^{\kappa,1}(x,t)|^2 dx + \kappa\int_0^t\int_{\mathbb{R}^3}|\nabla \bar{w}^{\kappa,1}(x,s)|^2dx ds\\ + \int_0^t\int_{\mathbb{R}^3}|\kappa{\rm div}\, \bar{w}^{\kappa,1}(x,s)|^2 dx ds \leq \int_0^t\int_{\mathbb{R}^3} |q^{\kappa}|^2 dx ds;
\end{multline*}
thus, thanks to \eqref{E2.24}, we find that
\begin{multline*}
    \kappa\int_{\mathbb{R}^3}|\bar{w}^{\kappa,1}(x,T)|^2 dx + \kappa\int_0^T\int_{\mathbb{R}^3}|\nabla \bar{w}^{\kappa,1}(x,s)|^2dx ds\\ + \int_0^T\int_{\mathbb{R}^3}|\kappa{\rm div}\, \bar{w}^{\kappa,1}(x,s)|^2 dx ds \leq c T^{\frac{1}{2}}\left(\|u_0\|^2_{L^{3,\infty}(\mathbb{R}^3)} + \|u_0\|^4_{L^{3,\infty}(\mathbb{R}^3)} \right)^2,
\end{multline*}
for all $T>0$. Since ${\rm div}\, z^{\kappa} = 0$, we finally obtain that
\begin{equation}\label{E2.26}
    \int_0^T\int_{\mathbb{R}^3}|\kappa{\rm div}\, w^{\kappa,1}(x,s)|^2 dx ds \leq c T^{\frac{1}{2}}\left(\|u_0\|^2_{L^{3,\infty}(\mathbb{R}^3)} + \|u_0\|^4_{L^{3,\infty}(\mathbb{R}^3)} \right)^2,
\end{equation}
for all $T>0$. 
Let us  mention that the trick to get the space-time uniform $L_2$-estimate for the term $\kappa {\rm div}\, w^{\kappa,1}$ is inspired by a similar estimate established in \cite{Lady95} for the stationary Lam\'e system and will be used once more in the next step.

\paragraph{Part II.}
We introduce the following functions
\[
\hat{w}^{\kappa} := \sum_{i=1}^3\hat{w}^{\kappa,i},\quad w^{\kappa,2} := w^{\kappa} - w^{\kappa,1} - \hat{w}^{\kappa}\quad\mbox{and }p^{\kappa} := \sum_{i=1}^3p^{\kappa}_i;
\]
we get that
\begin{equation}\label{E2.27}
    \left\{
    \begin{gathered}
    \partial_t w^{\kappa,2} - \Delta w^{\kappa,2} - \kappa \nabla {\rm div}\, w^{\kappa,2} =  \nabla p^{\kappa} \quad \mbox{in }\mathbb{R}^3\times \mathbb{R}_+ \\
    w^{\kappa,2}|_{t=0} = 0 \quad \mbox{in }\mathbb{R}^3,
    \end{gathered}
    \right.
\end{equation}
and again, by means reminiscent to what has been done in the proof of Theorem \ref{existenceglobalweaksol}, we find for all $t>0$ another energy estimate:
\begin{multline*}
    \frac{1}{2}\int_{\mathbb{R}^3}|w^{\kappa,2}(x,t)|^2 dx + \int_0^t\int_{\mathbb{R}^3}|\nabla w^{\kappa,2}(x,s)|^2dx ds\\ + \kappa\int_0^t\int_{\mathbb{R}^3}|{\rm div}\, w^{\kappa,2}(x,s)|^2 dx ds \leq \left(\int_0^t\int_{\mathbb{R}^3} |p^{\kappa}|^2 dx ds \right)^{\frac{1}{2}}\left(\int_0^t\int_{\mathbb{R}^3}|{\rm div}\, w^{\kappa,2}|^2 dx ds\right)^{\frac{1}{2}}.
\end{multline*}
Multiplying both sides of the previous inequality by $\kappa$ and using Young's inequality, we get
\begin{multline*}
    \kappa\int_{\mathbb{R}^3}|w^{\kappa,2}(x,t)|^2 dx + \kappa\int_0^t\int_{\mathbb{R}^3}|\nabla w^{\kappa,2}(x,s)|^2dx ds\\ + \int_0^t\int_{\mathbb{R}^3}|\kappa{\rm div}\, w^{\kappa,2}(x,s)|^2 dx ds \leq \int_0^t\int_{\mathbb{R}^3} |p^{\kappa}(x,s)|^2 dx ds,
\end{multline*}
and thanks to \eqref{E2.20}, we finally find that
\begin{multline}
    \kappa\int_{\mathbb{R}^3}|w^{\kappa,2}(x,T)|^2 dx + \kappa\int_0^T\int_{\mathbb{R}^3}|\nabla w^{\kappa,2}(x,t)|^2dx dt\\ + \int_0^T\int_{\mathbb{R}^3}|\kappa{\rm div}\, w^{\kappa,2}|^2 dx ds \leq cT^{\frac{1}{2}}\|u_0\|^2_{L^{3,\infty}(\mathbb{R}^3)}\left(\|u_0\|^2_{L^{3,\infty}(\mathbb{R}^3)} + \|u_0\|^4_{L^{3,\infty}(\mathbb{R}^3)} \right),
\end{multline}
for all $T>0$. By noticing that ${\rm div}\, \hat{w}^{\kappa} = 0$ and taking into account \eqref{E2.26}, we arrive at  estimate (\ref{important}).
\paragraph{Part III.} All that is left in order to finish the proof of (\ref{important}) is to justify the construction of the functions $w^{\kappa,1}$, $\hat{w}^{\kappa,i}$ and $p^{\kappa}_i$ ($i=1,2,3$) such that \eqref{E2.11} and  \eqref{E2.12} hold. For the functions $\hat{w}^{\kappa,i}$ and $p^{\kappa}_i$, the existence comes from  results on the inhomogeneous heat equation in the whole space together with estimates  established earlier for the terms:
$$
v\cdot\nabla w^{\kappa} + \frac{1}{2}{\rm div}\,(v\otimes w^{\kappa}),\quad \frac{1}{2} w^{\kappa}\cdot\nabla v\quad\mbox{and}\quad v\cdot\nabla v.
$$
Therefore, we skip the details for the sake of brevity.

Now for the construction of $w^{\kappa,1}$, we go back to the function $q^{\kappa}$ such that
$$-\Delta q^{\kappa} = {\rm div}\,(w^{\kappa}\cdot \nabla w^{\kappa} + \frac{w^{\kappa}}{2}{\rm div}\, w^{\kappa}), $$
and  introduce the function $Q^{\kappa} := \nabla q^{\kappa} + w^{\kappa}\cdot \nabla w^{\kappa} + \frac{1}{2}w^{\kappa}{\rm div}\, w^{\kappa}$ such that ${\rm div}\, Q^{\kappa} = 0$ and the following estimates holds (see \eqref{E2.22})
$$
\|Q^{\kappa}\|_{L_{\frac{6}{5},2}(Q_T)} + \|\nabla q^{\kappa}\|_{L_{\frac{6}{5},2} (Q_T)} < \infty,
$$
for all $T>0$.

Now, let us look for the solution to the Cauchy problem (\ref{E2.13}) in the form  $w^{\kappa,1}=w^{\kappa,11}+w^{\kappa,12}$, where
$$
\sup_{0<t<T}\|w^{\kappa,1i}(\cdot,t)\|^2_{L_2(\mathbb{R}^3)} + \int_0^T\int_{\mathbb{R}^3}|w^{\kappa,1i}(x,t)|^2 dx dt < \infty,\quad (i=1,2)
$$
for all $T>0$ and
\begin{equation*}
    \left\{
    \begin{gathered}
    \partial_t w^{\kappa,11} - \Delta w^{\kappa,11} =  Q^{\kappa} \mbox{ in }Q_+ \\
    w^{\kappa,11}|_{t=0} = 0 \quad \mbox{in }\mathbb{R}^3,
    \end{gathered}
    \right.\mbox{ and }
    \left\{
    \begin{gathered}
    \partial_t w^{\kappa,12} - (1+\kappa) \Delta w^{\kappa,12} =  \nabla q^{\kappa} \mbox{ in }Q_+ \\
    w^{\kappa,12}|_{t=0} = 0 \quad \mbox{in }\mathbb{R}^3.
    \end{gathered}
    \right.
\end{equation*}
Now, notice that ${\rm div}\, w^{\kappa,11} = $ and ${\rm curl}\, w^{\kappa,12} = 0$ in $Q_+$ and $\Delta w^{\kappa,12} = \nabla{\rm div}\, w^{\kappa,12}$ 
 and we are done by setting $w^{\kappa,1} := w^{\kappa,11} + w^{\kappa,12}$. This concludes the proof of (\ref{important}).

\begin{cor}\label{Coro1}
Let $U^{\kappa}$ and $W^{\kappa}$ be defined  in the proof of the previous theorem, see Part I. Then, the following estimates hold
\[
\|W^{\kappa}\|^2_{L_2(\mathbb{R}^3)} + \|\nabla W^{\kappa}\|^2_{L_2(\mathbb{R}^3)} + \kappa\|{\rm div}\, U^{\kappa}\|^2_{L_2(\mathbb{R}^3)} \leq c\left( \|u_0\|^2_{L^{3,\infty}(\mathbb{R}^3)} + \|u_0\|^4_{L^{3,\infty}(\mathbb{R}^3)} \right),
\]
and
$$
\|\kappa \,{\rm div}\, U^{\kappa}\|^2_{L_2(\mathbb{R}^3)} \leq c\left( \|u_0\|^2_{L^{3,\infty}(\mathbb{R}^3)} + \|u_0\|^4_{L^{3,\infty}(\mathbb{R}^3)} \right)^2,
$$
where $c>0$ is a universal constant.
\end{cor}

Now, our goal is to establish a uniform $H^2$-estimate (\ref{H2est})  (with respect to $\kappa$) for the forward self-similar profile $U^{\kappa}$. We recall once again that:
\begin{equation*}
    \begin{gathered}
    U^{\kappa} = V + W^{\kappa},~{\rm div}\, u_0 = 0,\\
    |\partial^{\alpha}W^{\kappa}(x)| \leq \frac{c(\kappa,u_0)}{(1 + |x|)^{3+|\alpha|}},\quad |\partial^{\alpha}V(x)| \leq \frac{c(u_0)}{(1 + |x|)^{1+|\alpha|}}.
    \end{gathered}
\end{equation*}
The above decay estimates make legitimate the upcoming integration by part. Let us also recall the following  known formulae:
\begin{equation}\label{Equ1.2}
    \begin{gathered}
    v\cdot\Delta v = -|\nabla v|^2 + \Delta\frac{|v|^2}{2},\quad \Delta v= \nabla {\rm div}\, v - {\rm curl}\,({\rm curl}\, v)\\
    {\rm div}\,(x\cdot\nabla v) = x\cdot \nabla {\rm div}\, v + {\rm div}\, v,\quad {\rm curl}\,(x\cdot \nabla v) =  x\cdot\nabla {\rm curl}\, v + {\rm curl}\, v.
    \end{gathered}
\end{equation}

We find from \eqref{E1.5} the following equation for $W^{\kappa}$:
\begin{multline}\label{Equ1.3}
    -\Delta W^{\kappa} - \kappa\nabla {\rm div}\, W^{\kappa} - \frac{x}{2}\cdot \nabla W^{\kappa} - \frac{W^{\kappa}}{2} + U^{\kappa}\cdot \nabla W^{\kappa} + \frac{U^{\kappa}}{2}{\rm div}\, W^{\kappa}\\ + U^{\kappa}\cdot \nabla V = 0\quad\mbox{in }\mathbb{R}^3.
\end{multline}
Now, let us multiply system \eqref{Equ1.3} by $\Delta W^{\kappa}$ and integrate over the whole space to obtain (thanks also to the formulas \eqref{Equ1.2}):
\begin{multline}\label{Equ1.4}
    \int_{\mathbb{R}^3}|\Delta W^{\kappa}|^2 dx + \kappa\int_{\mathbb{R}^3}|\nabla {\rm div}\, W^{\kappa}|^2 dx \leq \kappa\int_{\mathbb{R}^3} \nabla{\rm div}\, W^{\kappa}\cdot {\rm curl}\,({\rm curl}\, W^{\kappa}) dx\\ - \frac{1}{2}\int_{\mathbb{R}^3}(x\cdot \nabla W^{\kappa}) \cdot \nabla {\rm div}\, W^{\kappa} dx + \frac{1}{2}\int_{\mathbb{R}^3}(x\cdot \nabla W^{\kappa}) \cdot {\rm curl}\,({\rm curl}\, W^{\kappa}) dx\\ + \frac{1}{2}\int_{\mathbb{R}^3}|\nabla W^{\kappa}|^2 dx - \frac{1}{4}\int_{\mathbb{R}^3}\Delta |W^{\kappa}|^2 dx + c\||U^{\kappa}| |\nabla W^{\kappa}|\|_{L_2(\mathbb{R}^3)} \|\nabla^2 W^{\kappa}\|_{L_2(\mathbb{R}^3)}\\ + \|U^{\kappa}\|_{L_4(\mathbb{R}^3)} \|\nabla V\|_{L_4(\mathbb{R}^3)} \|\nabla W^k\|_{L^{2}(\mathbb{R}^3)}.
\end{multline}
Next, we have
$$
\||U^{\kappa}| |\nabla W^{\kappa}|\|_{L_2(\mathbb{R}^3)}
     \leq c\|U^{\kappa}\|
     _{L_{6}(\mathbb{R}^3)}\|\nabla W^{\kappa}\|_{L_3(\mathbb{R}^3)}$$$$
     \leq c(\|V\|_{L_{6}(\mathbb{R}^3)}+\|\nabla W^k\|_{L^{2}(\mathbb{R}^3)})\|\nabla W^k\|^\frac 12_{L^{2}(\mathbb{R}^3)}  \|\nabla^2 W^k\|^\frac 12_{L^{2}(\mathbb{R}^3)} \leq       $$
     $$\leq c(\|u_0\|_{L^{3,\infty}(\mathbb{R}^3)})\|\nabla^2 W^k\|^\frac 12_{L^{2}(\mathbb{R}^3)},$$
     $$
    \|\nabla V\|_{L_4(\mathbb{R}^3)} \leq c\|u_0\|_{L^{3,\infty}(\mathbb{R}^3)},
   $$
   and
   $$\|\nabla^2 W^{\kappa}\|_{L_2(\mathbb{R}^3)} \leq c\|\Delta W^{\kappa}\|_{L_2(\mathbb{R}^3)}.$$
   Using inequality \eqref{Equ1.4}, successive integration by parts, the above estimates and the energy estimates  for $W^{\kappa}$, we find
\begin{equation}
    \int_{\mathbb{R}^3}|\nabla^2 W^{\kappa}|^2 dx + \kappa\int_{\mathbb{R}^3}|\nabla {\rm div}\, W^{\kappa}|^2 dx \leq c(\|u_0\|_{L^{3,\infty}(\mathbb{R}^3)}).
\end{equation}

Finally, let us establish the uniform $L_2$-estimate for $\kappa \nabla {\rm div}\, W^{\kappa}$. To achieve this, we set
\[
F^{\kappa} := -\Delta W^{\kappa} - \frac{W^{\kappa}}{2} + U^{\kappa}\cdot \nabla W^{\kappa} + \frac{U^{\kappa}}{2}{\rm div}\, W^{\kappa}\\ + U^{\kappa}\cdot \nabla V,
\]
and see from the above computations that
\[
\|F^{\kappa}\|_{L_2(\mathbb{R}^3)} \leq c(\|u_0\|_{L^{3,\infty}(\mathbb{R}^3)}).
\]
Now, writing
\[
\kappa\nabla{\rm div}\, W^{\kappa} = -\frac{x}{2}\cdot \nabla W^{\kappa} + F^{\kappa},
\]
we get
\begin{align*}
    \phantom{{}\leq{}}
    \kappa^2\int_{\mathbb{R}^3}|\nabla {\rm div}\, W^{\kappa}|^2 dx &\leq -\frac{1}{2}\int_{\mathbb{R}^3}(x\cdot \nabla W^{\kappa})\cdot \kappa \nabla {\rm div}\, W^{\kappa} dx + \|F^{\kappa}\|_{L_2(\mathbb{R}^3)}\|\kappa \nabla {\rm div}\, W^{\kappa}\|_{L_2(\mathbb{R}^3)}\\
    &\leq -\frac{\kappa}{4}\int_{\mathbb{R}^3}({\rm div}\, W^{\kappa})^2 dx + \|F^{\kappa}\|_{L_2(\mathbb{R}^3)}\|\kappa \nabla {\rm div}\, W^{\kappa}\|_{L_2(\mathbb{R}^3)}.
\end{align*}
Using Cauchy's inequality, we obtain that
\begin{equation}
    \kappa^2\int_{\mathbb{R}^3}|\nabla {\rm div}\,W^{\kappa}|^2 dx \leq c(\|u_0\|_{L^{3,\infty}(\mathbb{R}^3)}).
\end{equation}

\setcounter{equation}{0}
\section{Proof of Theorem \ref{Convergence} }
From Theorem \ref{UniformEstimates}, it follows that there exists  subsequence (still denoted in the same way as $W^\kappa$) such that
\begin{equation}\label{E1.241}
    W^\kappa\rightharpoonup  W,\qquad \nabla W^\kappa\rightharpoonup \nabla W, \qquad\nabla^2 W^\kappa\rightharpoonup\nabla^2 W  \end{equation}
and
\begin{equation}
	\label{pressureconvergence2}
	\kappa{\rm div}\,w^k\rightharpoonup P\qquad \kappa\nabla {\rm div}\,w^k\rightharpoonup\nabla P\end{equation}
in $L_2(\mathbb R^3)$. We also can state
that

\begin{equation}\label{conv1}
	W^k\to W
\end{equation}
in $L_{4,loc}(\mathbb R^3)$ and a.e. in $\mathbb R^3$, and
\begin{equation}\label{maxbound}
\sup\limits_\kappa\sup_{x\in \mathbb R^3}|W^k(x)|<\infty.	
\end{equation}
Having the above convergence, it is easy to show that the limit functions $U$ and $P$ satisfy the profile equations (\ref{E1.8}).

Now, let us justify strong convergence. Let
 $\bar{U}^{\kappa} := U^{\kappa} - U = W^{\kappa} - W$. Next, we get from (\ref{Toy-Mod1}) and (\ref{E1.8})
\begin{multline}\label{Equ2.3}
    -\Delta \bar{U}^{\kappa} - \kappa \nabla {\rm div}\, \bar{U}^{\kappa} + \nabla P + U^{\kappa}\cdot \nabla \bar{U}^{\kappa} + \frac{\bar{U}^{\kappa}}{2}{\rm div}\, U^{\kappa} + \frac{U}{2}{\rm div}\, U^{\kappa}\\+ \bar{U}^{\kappa}\cdot \nabla U - \frac{x}{2}\cdot \nabla \bar{U}^{\kappa} - \frac{1}{2}\bar{U}^{\kappa} = 0\quad\mbox{in }\mathbb{R}^3.
\end{multline}
Multiplying the previous equation by $\bar{U}^{\kappa}$ and integrating on the whole of $\mathbb{R}^3$, we obtain (the latter can be verified by suitable cut-off and passing to the limit with the help of (\ref{E1.241})  and (\ref{pressureconvergence2})) that
\begin{multline*}
    \int_{\mathbb{R}^3}|\nabla \bar{U}^{\kappa}|^2 dx + \kappa\int_{\mathbb{R}^3}|{\rm div}\, \bar{U}^{\kappa}|^2 dx + \frac{1}{4}\int_{\mathbb{R}^3}|\bar{U}^{\kappa}|^2 dx \leq \frac{\|P\|_{L_2(\mathbb{R}^3)}}{\kappa}\|\kappa {\rm div}\, U^{\kappa}\|_{L_2(\mathbb{R}^3)}\\ + \frac{1}{2\kappa}\|U\|_{L_4(\mathbb{R}^3)}\|\bar{U}^{\kappa}\|_{L_4(\mathbb{R}^3)}\|\kappa {\rm div}\, U^{\kappa}\|_{L_2(\mathbb{R}^3)} + \int_{\mathbb{R}^3}|\bar{U}^{\kappa}|^2 |\nabla U| dx.
\end{multline*}

Since $L_4$-norm of $U$ and $\bar U^\kappa$
are uniformly bounded, it is enough to show that the last term on the right hand side of the above inequality tends to zero as $\kappa\to \infty$. Indeed,
$$\int_{\mathbb{R}^3}|\bar{U}^{\kappa}|^2 |\nabla U| dx\leq \Big(\int_{\mathbb R^3}|\bar{U}^{\kappa}|^2\Big)^\frac 12\Big(\int_{\mathbb R^3}|\bar{U}^{\kappa}|^2|\nabla U|^2\Big)^\frac 12.$$
The first factor on the right hand of the latter inequality is bounded while the second one tends to zero by Lebesgue theorem, see
(\ref{conv1}) and (\ref{maxbound}).

Now, let us notice the following identity
$$\sup_{0< \tau\leq t}\int_{\mathbb R^3}|u^\kappa(x,\tau)-u(x,\tau)|^2dx+\int^t_0\int_{\mathbb R^3}|\nabla(u^\kappa-u)|^2dxd\tau=
$$
$$=\sqrt t \int_{\mathbb R^3}(|\bar U^\kappa|^2+2|\nabla \bar U^\kappa|^2)dx,
$$ where $u(x,t)=\frac 1{\sqrt t}U(\frac x{\sqrt t})$.
 It is not so difficult to deduce from the above identity that in fact $u$ is a global weak $L^{3,\infty}$-solution to the Navier-Stokes system. To this end, one needs to take into account semigroup estimates and the above strong convergence.

 Now, since  $u(x,t)=\frac 1{\sqrt t}U(\frac x{\sqrt t})$ is a global weak $L^{3,\infty}$-solution to the Navier-Stokes system, estimates (\ref{DecayNavierStokes}) are true as well. It is shown in \cite{Jia14}.

 To show the strong convergence of $\nabla^2 \bar U^\kappa$ in $L_2(\mathbb R^3)$ to zero, it is sufficient to multiply equation
(\ref{Equ2.3}) by $\Delta \bar U^\kappa$ and use the strong convergence of $\bar U^\kappa$ and $\nabla \bar U^\kappa$, and the weak convergence of  $\nabla^2\bar U^\kappa$ in $L_2(\mathbb R^3)$ to zero.
\setcounter{equation}{0}
\section{Appendix I: Existence of Global Weak $L^{3,\infty}$- solutuions}

In this section, we are going to prove  Theorem \ref{existenceglobalweaksol}.
Let us start with recording the known fact about the decomposition in Lorentz spaces; for a proof, we refer to Lemma 3.1 in \cite{McCor13}.
\begin{lemma}\label{LA.1}
Take $1<t<r<s<\infty$, and suppose that $g\in L^{r,\infty}(\mathbb{R}^3)$. Then for any $N>0$, set $\bar{g}^N := g \mathbb{1}_{|g|\leq N} $ and $\hat{g}^N := g - \bar{g}^N$. Then
\begin{equation*}
    \|\bar{g}^N\|^s_{L_s(\mathbb{R}^3)} \leq \frac{s}{s - r}N^{s-r}\|g\|^r_{L^{r,\infty}(\mathbb{R}^3)} - N^s \left|\{x\in \mathbb{R}^3:|g(x)|>N\}\right|
\end{equation*}
and
\begin{equation*}
    \|\hat{g}^N\|^t_{L_t(\mathbb{R}^3)} \leq \frac{r}{r - t}N^{t-r}\|g\|^r_{L^{r,\infty}(\mathbb{R}^3)}.
\end{equation*}
\end{lemma}
We will only justify the estimate \eqref{E1.21} of the theorem because once we get this as a priori estimate, the machinery used to prove existence is fairly standard (see \cite{Bark18} for instance).

The proof is divided into two steps:

\paragraph{Step I.} Let $u^{\kappa}(x,t) = S_{\kappa}(t)u_0 + w^{\kappa}$ be a global weak $L^{3,\infty}-$solution to \eqref{Toy-Mod1} with initial data $u_0\in L^{3,\infty}(\mathbb{R}^3)$. We apply Lemma \ref{LA.1} to get $u_0 = \bar{u}_0^N + \hat{u}_0^N$ and we introduce the following functions
\begin{equation}\label{A.1}
    \bar{v}^N(x,t) := S_{\kappa}(t)\bar{u}_0^N(x),
\end{equation}

\begin{equation}\label{A.2}
    \hat{v}^N(x,t) := S_{\kappa}(t)\hat{u}_0^N(x)
\end{equation}
and
\begin{equation}\label{A.3}
    w^N(x,t) := w^{\kappa}(x,t) + \hat{v}^N(x,t) (= u^{\kappa}(x,t) - \bar{v}^N(x,t)),
\end{equation}
for all $(x,t)\in Q_+$ (here we omit the dependence of $\bar{v}^N$, $\hat{v}^N$ and $w^N$ with respect to $\kappa$ just for the sake of simplicity). By assumption and taking for instance $r=3,t=2$ in Lemma \ref{LA.1}, we have that
\[
\sup_{0<s<t}\|w^N(\cdot,s)\|_{L_2(\mathbb{R}^3)}^2 + \int_0^t\int_{\mathbb{R}^3}|\nabla w^N(x,s)|^2dx ds < \infty\quad\forall t>0,
\]
\begin{equation}\label{A.4}
    \lim_{t\to 0^+}\|w^N(\cdot,t) - \hat{u}_0^N\|_{L_2(\mathbb{R}^3)} = 0,
\end{equation}
and
\begin{equation}\label{A.5}
    \partial_t w^N - \Delta w^N - \kappa \nabla {\rm div}\, w^N + u^{\kappa}\cdot\nabla u^{\kappa} + \frac{u^{\kappa}}{2}{\rm div}\, u^{\kappa} = 0\mbox{ in }Q_+.
\end{equation}
From the construction of the semigroup $S_{\kappa}(t)$ in Proposition \ref{Prop1.1}, it follows that if $a\in L_s(\mathbb{R}^3)$ ($1<s<\infty$) then $\|S_{\kappa}(t)a - a\|_{L_s(\mathbb{R}^3)}\to 0$ as $t\to 0^+$; which gives us \eqref{A.4}.

Let $0\leq \varphi \in C^{\infty}_0(B)$ be such that $\varphi \equiv 1$ in $B(1/2)$ and $\varphi\equiv 0$ in $B\setminus B(3/4)$; we define, for every $R>0$, $\varphi_R(x) := \varphi(x/R)$. Now, from equation \eqref{A.5} and the definition \eqref{A.3}, we can get that for all $t>0$
\begin{multline}\label{A.6}
    \frac{1}{2}\int_{B(R)}|w^N(x,t)|^2 \varphi_R(x) dx + \int_0^t \int_{B(R)}|\nabla w^N|^2\varphi_R dx ds\\ + \kappa\int_0^t \int_{B(R)}|{\rm div}\, w^{\kappa}|^2 \varphi_R dx ds = \frac{1}{2}\int_{B(R)}|\hat{u}^N_0(x)|^2 \varphi_R dx + \frac{1}{2}\int_0^t\int_{B(R)}|w^N|^2 \Delta \varphi_R dx ds\\ - \kappa \int_0^t\int_{B(R)}{\rm div}\, w^N w^N\cdot \nabla \varphi_R dx ds + \frac{1}{2}\int_0^t\int_{B(R)}|w^N|^2 w^N\cdot\nabla \varphi_R dx ds\\
    +\int_0^t \int_{B(R)}\left((w^N\cdot\nabla w^N)\cdot \bar{v}^N-(\bar{v}^N\cdot \nabla w^N)\cdot w^N + \frac{1}{2}\bar{v}^N\cdot w^N{\rm div}\, w^N \right)\varphi_R dx ds\\ + \int_0^t \int_{B(R)} (\bar{v}^N\cdot\nabla w^N)\cdot \bar{v}^N\varphi_R dx ds + \int_0^t \int_{B(R)}w^N\cdot\nabla \varphi_R \bar{v}^N\cdot w^N dx ds\\ + \int_0^t \int_{B(R)} \bar{v}^N\cdot\nabla \varphi_R \bar{v}^N\cdot w^N dx ds;
\end{multline}
where, for simplicity, we write the right-hand side of the previous identity as follows
\[
\frac{1}{2}\int_{B(R)}|\hat{u}^N_0(x)|^2 \varphi_R dx + \sum_{k=1}^7 I_k(R).
\]
The aim now is to estimate the $I_k$'s and take the limit $R\to \infty$. To this end, we need the following known estimate:
\begin{equation}\label{A.7}
    \|a\|_{s,l,Q_T} \leq c(s,l) |a|_{2,Q_T},
\end{equation}
for $2\leq s \leq 6$ and $l$ satisfying
\[
\frac{3}{s} + \frac{2}{l} = \frac{3}{2};
\]
here $$|a|_{2,Q_T}=\left({\rm ess sup}_{0<t<T}\|a(\cdot,t)\|^2_{L_2(\Omega)} + \|\nabla a\|^2_{2,Q_T}\right)^\frac 12.$$

We have
\begin{multline*}
    I_1(R) + I_2(R) + I_3(R) \leq \frac{c t}{R^2}\sup_{0<s<t}\|w^N(\cdot,s)\|^2_{L_2(\mathbb{R}^3)}\\ + \frac{\kappa c t^{\frac{1}{2}}}{R}\sup_{0<s<t}\|w^N(\cdot,s)\|_{L_2(\mathbb{R}^3)}\left(\int_0^t\int_{\mathbb{R}^3}|\nabla w^N(x,s)|^2 dx ds\right)^{\frac{1}{2}}\\ + \frac{c t^{\frac{1}{2}}}{R}\left(\int_0^t\left(\int_{\mathbb{R}^3}|w^N(x,s)|^3 dx\right)^{\frac{4}{3}}ds \right)^{\frac{1}{2}} \to 0\mbox{ as }R\to \infty;
\end{multline*}
Next,
\[
I_4(R) \leq \frac{5}{2}\int_0^t\left(\int_{B(R)}|\nabla w^N|^2\varphi_R dx \right)^{\frac{1}{2}}\left(\int_{B(R)}|w^N\varphi^{\frac{1}{2}}_R|^3 dx\right)^{\frac{1}{3}}\left(\int_{\mathbb{R}^3}|\bar{v}^N|^6 dx\right)^{\frac{1}{6}}ds,
\]
but by interpolation, we find
\begin{align*}
\phantom{{}\leq{}}
    \|w^N\varphi^{\frac{1}{2}}\|_{L_3(B(R))} &\leq \|w^N\varphi^{\frac{1}{2}}\|^{\frac{1}{2}}_{L_2(B(R))}\|w^N\varphi^{\frac{1}{2}}\|^{\frac{1}{2}}_{L_6(B(R))}\\
    &\leq c \left(\int_{B(R)}|w^N|^2\varphi_R dx\right)^{\frac{1}{4}}\left[\left( \int_{B(R)}|\nabla w^N|^2\varphi_R dx \right)^{\frac{1}{4}}\right.\\&\mathrel{\phantom{=}}\left. + \left(\int_{B(R)}|\nabla \varphi_R^{\frac{1}{2}}|^2 |w^N|^2 dx\right)^{\frac{1}{4}} \right];
\end{align*}
consequently
\begin{align*}
    I_4(R) &\leq c \int_0^t\left(\int_{B(R)}|\nabla w^N|^2 \varphi_R dx\right)^{\frac{3}{4}}\left(\int_{B(R)} |w^N|^2\varphi_R dx \right)^{\frac{1}{4}}\left(\int_{\mathbb{R}^3}|\bar{v}^N|^6 dx\right)^{\frac{1}{6}}ds\\&\mathrel{\phantom{=}} + \frac{c}{R^{\frac{1}{2}}} \int_0^t\left(\int_{B(R)}|\nabla w^N|^2 \varphi_R dx\right)^{\frac{1}{2}}\left(\int_{B(R)} |w^N|^2dx\right)^{\frac{1}{2}}\left(\int_{\mathbb{R}^3}|\bar{v}^N|^6 dx\right)^{\frac{1}{6}}ds\\
    &\leq \epsilon_1\int_0^t\int_{B(R)}|\nabla w^N|^2\varphi_R dx + c(\epsilon_1)\int_0^t\left( \int_{\mathbb{R}^3}|\bar{v}^N|^6 dx\right)^{\frac{2}{3}}\left(\int_{B(R)}|w^N|^2\varphi_R dx \right)ds\\ &+ I^{(0)}_4(R),
\end{align*}
with $\epsilon_1>0$ and
\begin{align*}
    I^{(0)}_4(R) &:= \frac{c}{R^{\frac{1}{2}}} \int_0^t\left(\int_{B(R)}|\nabla w^N|^2 \varphi_R dx\right)^{\frac{1}{2}}\left(\int_{B(R)} |w^N|^2dx\right)^{\frac{1}{2}}\left(\int_{\mathbb{R}^3}|\bar{v}^N|^6 dx\right)^{\frac{1}{6}}ds\\ &\leq \frac{c t^{\frac{1}{2}}}{R^{\frac{1}{2}}}\|\bar{u}^N_0\|_{L_6(\mathbb{R}^3)}\sup_{0<s<t}\|w^N(\cdot,s)\|_{L_2(\mathbb{R}^3)}\left(\int_0^t\int_{\mathbb{R}^3}|\nabla w^N(x,s)|^2 dx ds\right)^{\frac{1}{2}}\\&\to 0\mbox{ as } R\to \infty,
\end{align*}
where Proposition \ref{Prop1.1} was used (with $s=s_1=6$) in the last inequality.

For $I_5(R)$, we have
\begin{align*}
    I_5(R) &\leq \epsilon_2 \int_0^t \int_{B(R)}|\nabla w^N|^2 \varphi_R dx + c(\epsilon_2)\int_0^t\int_{\mathbb{R}^3}|\bar{v}^N|^4dx ds\\
    &\leq \epsilon_2 \int_0^t \int_{B(R)}|\nabla w^N|^2 \varphi_R dx + c(\epsilon_2)N t \|u_0\|^3_{L^{3,\infty}(\mathbb{R}^3)},
\end{align*}
where we used Proposition \ref{Prop1.1} (with $s=s_1=4$) and Lemma \ref{LA.1} (with $s=4$ and $r=3$) in the last inequality.\\
Finally,
\begin{multline*}
    I_6(R) + I_7(R) \leq \frac{c t^{\frac{1}{4}}}{R^{\frac{1}{4}}}\|\bar{u}_0^N\|_{L_4(\mathbb{R}^3)}\left(\int_0^t\left(\int_{\mathbb{R}^3}|w^N|^4 dx\right)^{\frac{2}{3}}ds \right)^{\frac{3}{4}}\\ + \frac{c t}{R}\|\bar{u}_0^N\|^2_{L_4(\mathbb{R}^3)}\sup_{0<s<t}\|w^N\|_{L_2(\mathbb{R}^3)} \to 0 \mbox{ as }R\to \infty.
\end{multline*}
Summarising our efforts, we get (from \eqref{A.6}) that
\begin{multline*}
    \frac{1}{2}\int_{B(R)}|w^N(x,t)|^2 \varphi_R(x) dx + \int_0^t \int_{B(R)}|\nabla w^N|^2\varphi_R dx ds\\ + \kappa\int_0^t \int_{B(R)}|{\rm div}\, w^{\kappa}|^2\varphi_R dx ds \leq \frac{1}{2}\int_{B(R)}|\hat{u}^N_0(x)|^2 \varphi_R dx + (\epsilon_1 + \epsilon_2)\int_0^t \int_{B(R)}|\nabla w^N|^2\varphi_R dx ds\\ + c(\epsilon_1)\int_0^t\left( \int_{\mathbb{R}^3}|\bar{v}^N|^6 dx\right)^{\frac{2}{3}}\left(\int_{B(R)}|w^N|^2\varphi_R dx \right)ds + c(\epsilon_2)N t \|u_0\|^3_{L^{3,\infty}(\mathbb{R}^3)} + J(R),
\end{multline*}
with $J(R)\to 0$ as $R\to \infty$. Choosing suitably $\epsilon_1$ and $\epsilon_2$, and using Proposition \ref{Prop1.1} and Lemma \ref{LA.1}, in order to get
\[
\left( \int_{\mathbb{R}^3}|\bar{v}^N|^6 dx\right)^{\frac{2}{3}} \leq c N^2\|u_0\|^2_{L^{3,\infty}(\mathbb{R}^3)},
\]
we find that
\begin{multline*}
    \int_{B(R/2)}|w^N(x,t)|^2 dx + \int_0^t \int_{B(R/2)}|\nabla w^N|^2 dx ds\\ + \kappa\int_0^t \int_{B(R/2)}|{\rm div}\, w^{\kappa}|^2 dx ds \leq \int_{B(R)}|\hat{u}^N_0(x)|^2 dx +\\ c\left(N^2\|u_0\|^2_{L^{3,\infty}(\mathbb{R}^3)}\int_0^t\int_{B(R)}|w^N(x,s)|^2dx ds + Nt\|u_0\|^3_{L^{3,\infty}(\mathbb{R}^3)}\right) + 2 J(R).
\end{multline*}
Finally, taking the limit $R\to \infty$ in the above inequality, 
we get
\begin{multline}\label{A.8}
    \int_{\mathbb{R}^3}|w^N(x,t)|^2 dx + \int_0^t \int_{\mathbb{R}^3}|\nabla w^N|^2 dx ds\\ + \kappa\int_0^t \int_{\mathbb{R}^3}|{\rm div}\, w^{\kappa}|^2 dx ds \leq \int_{\mathbb{R}^3}|\hat{u}^N_0(x)|^2 dx +\\ c\left(N^2\|u_0\|^2_{L^{3,\infty}(\mathbb{R}^3)}\int_0^t\int_{\mathbb{R}^3}|w^N(x,s)|^2dx ds + Nt\|u_0\|^3_{L^{3,\infty}(\mathbb{R}^3)}\right),
\end{multline}
for all $t>0$. By Applying Gronwall's lemma 
to \eqref{A.8}, we obtain
\begin{equation}\label{A.9}
    \int_{\mathbb{R}^3}|w^N(x,t)|^2 dx \leq \left(\int_{\mathbb{R}^3}|\hat{u}^N_0(x)|^2 dx + N^{-1}\|u_0\|_{L^{3,\infty}(\mathbb{R}^3)}\right)\exp(cN^2t\|u_0\|^2_{L^{3,\infty}(\mathbb{R}^3)}),
\end{equation}
for all $t>0$. Now, by substituting \eqref{A.9} in \eqref{A.8} and using the fact that
\[
\int_{\mathbb{R}^3}|\hat{u}^N_0(x)|^2 dx \leq c N^{-1}\|u_0\|^3_{L^{3,\infty}(\mathbb{R}^3)},
\]
and, we obtain
\begin{multline}
    \int_{\mathbb{R}^3}|w^N(x,t)|^2 dx + \int_0^t \int_{\mathbb{R}^3}|\nabla w^N|^2 dx ds\\ + \kappa\int_0^t \int_{\mathbb{R}^3}|{\rm div}\, w^{\kappa}|^2 dx ds \leq c N^{-1}\left(\|u_0\|_{L^{3,\infty}(\mathbb{R}^3)} + \|u_0\|^3_{L^{3,\infty}(\mathbb{R}^3)} \right)\exp(cN^2t\|u_0\|^2_{L^{3,\infty}(\mathbb{R}^3)})\\ + c N t \|u_0\|^3_{L^{3,\infty}(\mathbb{R}^3)},
\end{multline}
for all $t>0$ and $N>0$.

\paragraph{Step II.}
Firstly, let us notice that
\[
\|\hat{v}^N(\cdot,t)\|^2_{L_2(\mathbb{R}^3)} + 2\int_0^t\int_{\mathbb{R}^3}\left[|\nabla \hat{v}^N|^2 + \kappa({\rm div}\, \hat{v}^N)^2\right]dx ds = \|\hat{u}_0^N\|^2_{L_2(\mathbb{R}^3)} (\leq c N^{-1}\|u_0\|^3_{L^{3,\infty}(\mathbb{R}^3)}),
\]
for all $t>0$. Secondly, going back to the definition of $w^N$ (see \eqref{A.3}) and using the above identity, we see that
\begin{multline}
    \int_{\mathbb{R}^3}|w^{\kappa}(x,t)|^2 dx + \int_0^t \int_{\mathbb{R}^3}|\nabla w^{\kappa}|^2 dx ds\\ + \kappa\int_0^t \int_{\mathbb{R}^3}|{\rm div}\, w^{\kappa}|^2 dx ds \leq c N^{-1}\left(\|u_0\|_{L^{3,\infty}(\mathbb{R}^3)} + \|u_0\|^3_{L^{3,\infty}(\mathbb{R}^3)} \right)\exp(cN^2t\|u_0\|^2_{L^{3,\infty}(\mathbb{R}^3)})\\ + c N t \|u_0\|^3_{L^{3,\infty}(\mathbb{R}^3)},
\end{multline}
for all $t>0$ and $N>0$. Selecting  now
\[
N = \frac{1}{\sqrt{2 c t \|u_0\|^2_{L^{3,\infty}(\mathbb{R}^3)}}},
\]
we finally find
\begin{multline}
    \int_{\mathbb{R}^3}|w^{\kappa}(x,t)|^2 dx + \int_0^t \int_{\mathbb{R}^3}|\nabla w^{\kappa}|^2 dx ds\\ + \kappa\int_0^t \int_{\mathbb{R}^3}|{\rm div}\, w^{\kappa}|^2 dx ds \leq c_0t^{\frac{1}{2}}\left(\|u_0\|^2_{L^{3,\infty}(\mathbb{R}^3)} + \|u_0\|^4_{L^{3,\infty}(\mathbb{R}^3)} \right),
\end{multline}
for all $t>0$ and a universal constant $c_0>0$. The same machinery works for system \eqref{Toy-Mod2}. This concludes the proof.

\setcounter{equation}{0}\section{Appendix II: Existence of Forward Self-Similar Solutions}
The following two results are needed in the proof of the existence of forward self-similar solutions for our models. Our setting is as follows: we take an initial data $u_0 = (u_0^1,u_0^2,u_0^3)$ which is a $(-1)$-homogeneous vector field such that $u_0|_{\partial B_1} \in C^{\infty}(\partial B_1)$. In this case, one can steadily show that
\[
|\partial^{\alpha} u_0(x)| \leq \frac{C(\alpha,u_0)}{|x|^{1+|\alpha|}},\quad \forall \alpha\in \mathbb{N}^3.
\]
We have the following decay estimate.
\begin{theorem}[A priori estimate for forward self-similar solutions]\label{ThmB.1}
Let $u_0$ as above and $u$ be a scale invariant global weak $L^{3,\infty}-$solution to system \eqref{Toy-Mod1} or \eqref{Toy-Mod2}. Then, the solution profile $U(\cdot) (:=u(\cdot,1))$  belongs to $C^{\infty}(\mathbb{R}^3)$ and
\[
|\partial^{\alpha}(U - S_{\kappa}(1)u_0)(x)| \leq \frac{C(\alpha,\kappa,u_0)}{(1 + |x|)^{3 + |\alpha|}},
\]
for all $\alpha\in \mathbb{N}^3$ (with $|\alpha| = \alpha_1 + \alpha_2 + \alpha_3$).
\end{theorem}
\begin{proof}
The proof of this theorem follows the same ideas as the proof of a similar result obtained in \cite{Jia14} for the incompressible Navier-Stokes system (the difference here being that the leading term in our system is the Lam\'e operator). It requires a lot of technical intermediate results which are not the point of this work. However, for the reader's convenience we outline the proof here; see \cite{Hou20} for a detailed proof. Following \cite{Jia14}, we should investigate local regularity of our solutions at the initial moment of time; this leads to the statement:


\begin{pro}Let $u_0 \in L^{3,\infty}(\mathbb{R}^3)$. Suppose in addition that $M:= \|u_0\|_{C^{\gamma}(B(2))}<\infty$. Then, there exists a positive time $T=T(\kappa,M,\|u_0\|_{L^{3,\infty}(\mathbb{R}^3)})$ such that any global weak $L^{3,\infty}-$solution $u$ to \eqref{Toy-Mod1} or \eqref{Toy-Mod2} satisfies:
$$
\|u\|_{C^{\gamma,\frac{\gamma}{2}}(\overline{B(1/4)}\times [0,T])} \leq C(\gamma,\kappa,M,\|u_0\|_{L^{3,\infty}(\mathbb{R}^3)}).
$$\end{pro}
The idea behind the proof of this statement is as follows: we localise the initial data $u_0$ in $B$ and then solves system \eqref{Toy-Mod1} or system \eqref{Toy-Mod2} for this localised initial data. The resulting solution $a$ is smooth and we proceed to show that the difference $u-a$, which is now null at $t=0$, remains regular near the initial time by $\epsilon-$regularity. We give the details of this claim's proof in \cite{Hou20}.

We know that
$
u(x,t) = \frac{1}{\sqrt{t}}U\left( \frac{x}{\sqrt{t}} \right)$, $t>0$,
where $U(\cdot) = S_{\kappa}(1)u_0 + W(\cdot,1)$, and estimates for $W$ are given by Corollary \ref{Coro1}. Thus, it's not too difficult to see that
\begin{equation}\label{B.1}
    \sqrt{t^*}\int_{B(1/\sqrt{t^*})} |U(y)|^2 dy + \sqrt{t^*}\int_{B(1/\sqrt{t^*})} |\nabla U(y)|^2 dy \leq C(\kappa, \|u_0\|_{C(\partial B)})(1 + \sqrt{t^*}),
\end{equation}
for all $t^*>0$.

On the other hand, for all $|x_0| = 8$, we have $u_0 \in C^{\infty}(B(x_0,4))$. Therefore, by the above proposition and by some simple bootstrapping arguments (with estimate \eqref{E1.21} at hand for the case of system \eqref{Toy-Mod1}), we have that there exists $T_2 = T_2(\kappa,u_0)>0$ such that
\begin{equation}\label{B.2}
    \|\partial_t \partial^{\alpha} u\|_{L_{\infty}(\overline{B(x_0,1/8)}\times [0,T_2])} \leq C(\alpha,\kappa,u_0),
\end{equation}
for all global weak $L^{3,\infty}-$solution $u$ to system \eqref{Toy-Mod1} or system \eqref{Toy-Mod2} with initial data $u_0$.
Since, for all $\lambda>0$, the scaled function $u^{\lambda}(x,t)= \lambda u(\lambda x,\lambda^2 t)$ is also a global weak $L^{3,\infty}-$solution to \eqref{Toy-Mod1} or \eqref{Toy-Mod2} with initial data $u_0$, then \eqref{B.2} holds also for $u^{\lambda}$ and we find
\[
|\lambda^{1+|\alpha|}\partial^{\alpha} u(\lambda x_0,\lambda^2 t) - \partial^{\alpha}u_0(x_0)| \leq C(\alpha,\kappa,u_0)t.
\]
Setting $y=x_0/\sqrt{t}$, and by using the homogeneity of $\partial^{\alpha} u_0$, we have:
\begin{equation}
    |\partial^{\alpha}(U - u_0)(y)| \leq \frac{C(\alpha,\kappa,u_0)}{|y|^{3+|\alpha|}},\quad \forall |y|> \frac{8}{\sqrt{T_2}}.
\end{equation}
Now, we choose $t^* = t^*(\kappa,u_0)$ in \eqref{B.1} sufficiently small so that
\begin{equation}\label{B.4}
    \int_{B(\frac{16}{\sqrt{T_2}})}\left(|U(y)|^2 + |\nabla U(y)|^2 dy \right) \leq C(\kappa,u_0).
\end{equation}
Since the profile $U$ satisfies either \eqref{E1.5} or \eqref{E1.6}, elliptic theory estimates give us:
\begin{equation}\label{B.5}
    \|U\|_{C^k(\overline{B(9/\sqrt{T_2})})} \leq C(\kappa,k,u_0)\quad (k=0,1,2\ldots)
\end{equation}
Going back to the definition of the semigroup $S_{\kappa}(t)$ in Proposition \ref{Prop1.1}, we obtain that
\[
\|\partial^{\alpha} S_{\kappa}(1)u_0\|_{L_{\infty}(\mathbb{R}^3)} \leq C(\alpha,\kappa,u_0)
\]
and
$$
|\partial^{\alpha}(S_{\kappa}(1)u_0 - u_0)(x)| \leq |\partial^{\alpha}(S(1)u_0^{(0)} - u_0^{(0)})(x)| + $$$$+|\partial^{\alpha}(S(1+\kappa)u_0^{(1)} - u_0^{(1)})(x)|\leq \frac{C(\alpha,\kappa,u_0)}{|x|^{3+|\alpha|}}
$$
by the known properties of the heat equation. This concludes the proof of the theorem.
\end{proof}

Another  important step in the proof of the main result of this section is as follows.
\begin{pro}[Decay for the linearly singularly forced Lam\'e system]\label{PropB.1}
Let $f \in C(\mathbb{R}^3)$ and suppose that $w \in L_{\infty}(0,T;L_{\gamma}(\mathbb{R}^3))$ for any $T>0$ and for some $\gamma \geq 1$ and moreover
\begin{equation}\label{B.6}
    \left\{
    \begin{gathered}
    \partial_t w - \Delta w - \kappa\nabla {\rm div}\, w = t^{-\frac{3}{2}}f(\frac{x}{\sqrt{t}})\quad \mbox{in } Q_+\\
    \lim_{t\to 0^+}\|w(\cdot,t)\|_{L_{\gamma}(\mathbb{R}^3)} = 0.
    \end{gathered}
    \right.
\end{equation}
Then:\\
    (i) If $\tilde{w}$ satisfies also the above conditions, then $\tilde{w} = w$. Consequently $w(\lambda x,\lambda^2 t) = w(x,t)$ for all $\lambda>0$.\\
    (ii) If $M:= \sup_{x\in \mathbb{R}^3}(1 + |x|)^3|f(x)| < \infty$, then by setting $W(x) = w(x,1)$ we get: $\|W\|_{C^{1 + \alpha}(B(R))}\leq c(\alpha,\kappa,R)M$ for $\alpha\in (0,1)$ and
    \[
    \sup_{x\in \mathbb{R}^3}\left[ (1+|x|)^2|W(x)| + (1+|x|)^3|\nabla W(x)| \right] \leq C(\kappa)M.
    \]
    (iii) Similarly, if $M:= \sup_{x\in \mathbb{R}^3}(1 + |x|)^4|f(x)| < \infty$, then
    \[
    \sup_{x\in \mathbb{R}^3}\left[ (1+|x|)^3|W(x)| + (1+|x|)^4|\nabla W(x)|\right] \leq C(\kappa)M.
    \]
\end{pro}
\begin{proof}
First, we observe that if a function $w \in L_{\infty}(0,T;L_{\gamma_1}(\mathbb{R}^3) + L_{\gamma_2}(\mathbb{R}^3))$ (for all $T>0$) is such that $\lim_{t\to 0^+}\|w(\cdot,t)\|_{L_{\gamma_1}(\mathbb{R}^3) + L_{\gamma_2}(\mathbb{R}^3)} = 0$ and
$$
\partial_t w - \Delta w - \kappa\nabla {\rm div}\, w = 0\quad \mbox{in }Q_+,
$$
then $w\equiv 0$. 

Second, we start with establishing a decay estimate for $\nabla {\rm div}\, w$. For this, notice that
\begin{equation}\label{B.7}
    {\rm div}\, w(x,t) = \int_0^t \int_{\mathbb{R}^3}\nabla \Gamma_{\kappa}(x-y,t-s)\cdot f(\frac{y}{\sqrt{s}})s^{-\frac{3}{2}}dy ds,
\end{equation}
with
\[
\Gamma_{\kappa}(x,t) = \frac{1}{[4\pi(1+\kappa)t]^{\frac{3}{2}}}\exp\left(-\frac{|x|^2}{4(1+\kappa)t}\right).
\]
A simple computation gives us
\begin{equation*}
    |\nabla{\rm div}\, W(x)| \leq c(\kappa)M\int_0^1\int_{\mathbb{R}^3}\frac{1}{\left( |x-y| + \sqrt{1-s} \right)^4}\frac{1}{\left( |y| + \sqrt{s} \right)^3} dy ds\leq C(\kappa)M|x|^{-3}
\end{equation*}
for $|x|>8$. Since $W$ satisfies the following system
\begin{equation}\label{B.8}
    -\Delta W - \kappa\nabla {\rm div}\, W - \frac{x}{2}\cdot \nabla W -\frac{W}{2} = f\quad \mbox{in }\mathbb{R}^3,
\end{equation}
and that (see for instance \eqref{B.7} combined with known estimates for the volume heat potential)
\[
\|{\rm div}\, W\|_{L_2(\mathbb{R}^3)} + \|\nabla {\rm div}\, W\|_{L_{\frac{3}{2}}(\mathbb{R}^3)} \leq C(\kappa)M,
\]
therefore, elliptic estimates for the equation
$$-(1+\kappa)\Delta {\rm div}\,W-x\cdot\nabla {\rm div}\,W-\frac 32{\rm div}\,W={\rm div}\,f$$
guarantee the estimate
\[
\|\nabla {\rm div}\, W\|_{L_{\infty}(B(12))} \leq C(\kappa)M.
\]
Consequently, we have
$$
\sup_{x\in \mathbb{R}^3}\left[ (1 + |x|)^3|\nabla {\rm div}\, W(x)| \right] \leq C(\kappa)M.
$$

Now, if we set $g(x) := f(x) + \nabla {\rm div}\, W(x)$, then
\[
\partial_t w - \Delta w = t^{-\frac{3}{2}}g(\frac{x}{\sqrt{t}})\quad \mbox{in }Q_+
\]
and thus
\[
w(\cdot,t) = \int_0^t \int_{\mathbb R^3}\Gamma (x-y,t)g(\frac{y}{\sqrt{s}})s^{-\frac{3}{2}} ds.
\]
And as previously, we can show that
$$
|W(x)| \leq c(\kappa)M\int_0^1\int_{\mathbb{R}^3}\frac{1}{\left( |x-y| + \sqrt{1-s} \right)^3}\frac{1}{\left( |y| + \sqrt{s} \right)^3} dy ds\leq $$$$\leq C(\kappa)M|x|^{-3}\log|x|,
$$
and
\[
|\nabla W(x)| \leq c(\kappa)M\int_0^1\int_{\mathbb{R}^3}\frac{1}{\left( |x-y| + \sqrt{1-s} \right)^4}\frac{1}{\left( |y| + \sqrt{s} \right)^3} dy ds\leq C(\kappa)M|x|^{-3},
\]
for $|x|>8$. The continuity estimates in $B(12)$ follow from standard elliptic estimates for \eqref{B.8}.

Third, it is proved by using the exact same ideas as in the previous point; the difference here being that the source term has a faster decay at infinity (which make things easier in this case).
\end{proof}

 Now, we are able to give a proof of the  existence of a scale invariant global weak $L^{3,\infty}-$weak solution to our models \eqref{Toy-Mod1} and \eqref{Toy-Mod2}. The proof is based on Leray-Schauder degree theory applied in a suitable function space in order to establish the existence of a solution to systems \eqref{E1.5} or \eqref{E1.6}.
\begin{proof}[Proof of Theorem \ref{existenceself-similar}]
We introduce the following function space
\begin{equation}\label{E2.1}
    X =\left\{ V\in C^1(\mathbb{R}^3): \sup_{x\in \mathbb{R}^3}\left[ (1+|x|)^2|V(x)| + (1+|x|)^3|\nabla V(x)| \right] < \infty \right\}
\end{equation}
endowed with the natural norm
\begin{equation}\label{E2.2}
    \|V\|_X = \sup_{x\in \mathbb{R}^3}\left[ (1+|x|)^2|V(x)| + (1+|x|)^3|\nabla V(x)| \right];
\end{equation}
the choice of this functional space is motivated by Theorem \ref{ThmB.1} and a need for compactness as we shall see below.

We are going to use the same notations as in the proof of Proposition \ref{Prop1.1}. Because of the scaling symmetry of $u_0$, we get that $u^{(0)}_0$ and $u^{(1)}_0$ are also $(-1)-$homogeneous. Moreover, elliptic estimates guarantee that $u^{(1)}_0,u^{(0)}_0\in C^{\infty}(\partial B)$ and we have
\[
|\partial^{\alpha} u^{(1)}_0(x)| + |\partial^{\alpha} u^{(0)}_0(x)| \leq \frac{C(\alpha,u_0)}{|x|^{1+|\alpha|}}.
\]
Consequently,
\[
|\partial^{\alpha} S_\kappa(1)u_0(x)| \leq |\partial^{\alpha}v^0(x,1)| + |\partial^{\alpha}v^1(x,1)|\leq \frac{C(\alpha,\kappa,u_0)}{(1+|x|)^{1+|\alpha|}}
\]
by the properties of the heat equation.

Next, introduce a parameter $\mu\in [0,1]$.
Let us consider the following problem: find $U$ such that

\begin{equation}\label{E2.3}
    -\Delta U - \kappa \nabla {\rm div}\, U + U\cdot\nabla U + \frac{U}{2}{\rm div}\, U - \frac{x}{2}\cdot \nabla U - \frac{U}{2} = 0\quad\mbox{in }\mathbb{R}^3,
\end{equation}
and $|U(x) - V_\mu| = o(|x|^{-1})$ as $|x|\to \infty$, where $V_\mu(x)=S_\kappa(1)(\mu u_0)(x)$. We will seek $U$ in the form
\begin{equation}
    U = V_{\mu} + W,\quad \mbox{where } W\in X.
\end{equation}
It is clear that $u(x,t) = \frac{1}{\sqrt{t}}U(\frac{x}{\sqrt{t}})$ is a global weak $L^{3,\infty}-$solution to \eqref{Toy-Mod1} with initial data $\mu u_0$ if and only if $U(x)$ satisfies the elliptic system \eqref{E2.3} and $U(x) = V_{\mu} + W$ for some $W\in X$ (by Theorem \eqref{ThmB.1}). Thus, we have reduced the problem to finding $W\in X$ such that
\begin{multline}\label{E2.5}
    -\Delta W - \kappa \nabla {\rm div}\, W - \frac{x}{2}\cdot \nabla W - \frac{W}{2} = - W\cdot \nabla W - V_{\mu}\cdot \nabla W - W\cdot\nabla V_{\mu} - V_{\mu}\cdot \nabla V_{\mu}\\-\frac{W}{2}{\rm div}\, W - \frac{V_{\mu}}{2}{\rm div}\, W - \frac{W}{2}{\rm div}\, V_{\mu} - \frac{V_{\mu}}{2}{\rm div}\, V_{\mu}\quad\mbox{in }\mathbb{R}^3.
\end{multline}
Notice that if we set $w(x,t) := \frac{1}{\sqrt{t}}W(\frac{x}{\sqrt{t}})$, we have that
\begin{equation}\label{E2.6}
    \left\{
    \begin{gathered}
    \partial_t w - \Delta w - \kappa\nabla {\rm div}\, w = t^{-\frac{3}{2}}F(\frac{x}{\sqrt{t}})\quad \mbox{in }Q_+\\
    w|_{t=0} = 0\quad\mbox{in }\mathbb{R}^3,
    \end{gathered}
    \right.
\end{equation}
where
\begin{multline*}
    F = - W\cdot \nabla W - V_{\mu}\cdot \nabla W - W\cdot\nabla V_{\mu} - V_{\mu}\cdot \nabla V_{\mu}\\-\frac{W}{2}{\rm div}\, W - \frac{V_{\mu}}{2}{\rm div}\, W - \frac{W}{2}{\rm div}\, V_{\mu} - \frac{V_{\mu}}{2}{\rm div}\, V_{\mu}
\end{multline*}
has the decay properties as in Proposition \ref{PropB.1} provided $W\in X$. Conversely, for a function $F$ with the decay estimates as in Proposition \ref{PropB.1}, system \eqref{E2.6} is uniquely solvable and we denote the solution profile at time $t=1$ as $\mathcal{G}(F)\in X$, i.e., $\mathcal{G}(F)(x):=w(x,1)$. Obviously, $\mathcal G$ is a linear operator. The latter allows as to reformulate \eqref{E2.5} as follows:
\begin{multline}
    \mbox{Find }W\in X \mbox{ such that } W = \mathcal{G}(W\cdot \nabla W - V_{\mu}\cdot \nabla W - W\cdot\nabla V_{\mu} - V_{\mu}\cdot \nabla V_{\mu}\\-\frac{W}{2}{\rm div}\, W - \frac{V_{\mu}}{2}{\rm div}\, W - \frac{W}{2}{\rm div}\, V_{\mu} - \frac{V_{\mu}}{2}{\rm div}\, V_{\mu})
\end{multline}
Now, let us define an operator $K: X\times [0,1]\to X$ be defined as: $\forall V\in X$, $\mu\in [0,1]$,
\begin{multline}\label{E2.8}
    K(V,\mu) = \mathcal{G}(V_{\mu}\cdot \nabla V_{\mu} + \frac{V_{\mu}}{2}{\rm div}\, V_{\mu}) + \mathcal{G}(V\cdot \nabla V + V_{\mu}\cdot \nabla V + V\cdot\nabla V_{\mu}\\ + \frac{V}{2}{\rm div}\, V + \frac{V_{\mu}}{2}{\rm div}\, V + \frac{V}{2}{\rm div}\, V_{\mu})
\end{multline}
Notice that $\mathcal{G}(V_{\mu}\cdot \nabla V_{\mu} + \frac{V_{\mu}}{2}{\rm div}\, V_{\mu}) = \mu^2 \mathcal{G}(V\cdot \nabla V + \frac{V}{2}{\rm div}\, V)$ has a one-dimensional range (thus is compact). 
In order to see that the second term is compact, let us consider a bounded sequence $V^{(j)}$ in $X$ together with $\mu_j \in [0,1]$ and set
\begin{multline*}
    G^{(j)} = \mathcal{G}(V^{(j)}\cdot \nabla V^{(j)} + V_{\mu_j}\cdot \nabla V^{(j)} + V^{(j)}\cdot\nabla V_{\mu_j} + \frac{V^{(j)}}{2}{\rm div}\, V^{(j)} + \frac{V_{\mu_j}}{2}{\rm div}\, V^{(j)}\\ + \frac{V^{(j)}}{2}{\rm div}\, V_{\mu_j})
\end{multline*}
The arguments of the operator $\mathcal{G}$ in the above formula having a decay $(1+|x|)^{-4}$ or better, which is uniform in $j$. From Proposition \ref{PropB.1}, we find that
\begin{equation*}
    \begin{gathered}
    \sup_j\|G^{(j)}\|_{C^{1+\alpha}(B(R))}<\infty,\quad \forall R>0\\
    \sup_j \sup_{x\in \mathbb{R}^3}\left[ (1 + |x|)^3|G^{(j)}(x)| + (1 + |x|)^4|\nabla G^{(j)}(x)| \right] < \infty,
    \end{gathered}
\end{equation*}
which implies the desired compactness in $X$ by the known arguments. The continuity follows the exact same arguments as for the compactness. Consequently, to solve the problem at hand i.e.
\begin{equation}
    \mbox{Find }W \in X \mbox{ such that } W + K(W,\mu) = 0, \qquad\mu\in [0,1],
\end{equation}
we can apply Leray-Schauder theory (see e.g. \cite{Maw99}). All the required a priori estimates are given by Theorem \ref{ThmB.1} and Proposition \ref{PropB.1}, thus the only thing to be verified is the solvability of the problem for small enough $\mu$. But this can be easily done by an application of the implicit function theorem to our functional; we skip the details here for the sake of brevity. Same reasoning for the model \eqref{Toy-Mod2}. And this concludes the proof.
\end{proof}
\paragraph{Acknowledgement}
The first author is supported by the Engineering and Physical Sciences Research Council [EP/L015811/1]. The second author is supported by the grant RFBR 20-01-00397.


\newpage
\begin{thebibliography}{}
\bibitem{Bark18}
T. Barker, G. Seregin, V. \v{S}ver\'{a}k (2018). On stability of weak Navier-Stokes solutions with large $L_{3,\infty}$ initial data. \sl{Commun. Partial Differ. Equ}, \bf{43}, \rm{no 4}, 628-651.

\bibitem{Gui17}
J. Guillod, V. \v{S}ver\'{a}k (2017).
Numerical investigations of non-uniqueness for the Naier-Stokes intial value problem in borderline spaces, {\sl{arxiv.org. 1704.00560}}.

\bibitem{Grafa08}
L. Grafakos (2008). Classical Fourier Analysis. {\sl{Springer-Verlag New York}}, {\bf{249}}.

\bibitem{Hou20}
F. Hounkpe (2020). Decay Estimate for some Toy-models related to the Navier-Stokes system. {\sl{To appear}}.

\bibitem{Jia14}
H. Jia, V. \v{S}ver\'{a}k (2014). Local-in-space estimates near initial time for weak solutions of the Navier-Stokes equations and forward self-similar solutions. {\sl{Invent math}}, {\bf{196}}, 233-265.

\bibitem{KiSer2007}
Kikuchi, N., Seregin, G. Weak solutions to the Cauchy problem for the Navier-Stokes equations satisfying the local energy inequality. Nonlinear equations and spectral theory, 141–164, Amer. Math. Soc. Transl. Ser. 2, 220, Adv. Math. Sci., 59, Amer. Math. Soc., Providence, RI, 2007.

\bibitem{Lady95}
O.A. Ladyzhenskaya, G. Seregin (1995). On one method of approximation of initial boundary value problems for the Navier-Stokes equations. \sl{J Math Sci}, \bf{75}, \rm{no 6}, 2038-2057.

\bibitem{LemRie2002} Lemarie-Rieusset, P. G. Recent developments in the Navier-Stokes problem.
Chapman and Hall/CRC
Research Notes in Mathematics, 431. Chapman and Hall/CRC, Boca Raton, FL, 2002. xiv+395 pp.

\bibitem{LemRie2016}
Lemarie-Rieusset, Pierre Gilles The Navier-Stokes problem in the 21st century. CRC Press, Boca Raton, FL, 2016. xxii+718 pp.

\bibitem{Leray1934}
Leray, J. (1934). Sur le mouvement d’un liquide visqueux emplissant l’espace. Acta Math. 63:
193--248.

\bibitem{Maw99}
J. Mawhin (1999). Leray-Schauder degree: a half-century of extensions and applications. {\sl{Topol. Methods Nonlinear Anal.}}, {\bf{14}}, 195-228.

\bibitem{McCor13}
D. S. McCormick, J. C. Robinson, J. L. Rodrigo (2013). Generalised Gagliardo-Nirenberg Inequalities Using Weak Lebesgue Spaces and $BMO$. {\sl{Milan J. Math.}}, {\bf{81}}, 265-289.
\end{thebibliography}
\end{document}